\documentclass[11pt]{amsart}

\usepackage{latexsym}
\usepackage{amssymb}
\usepackage{amsfonts}
\usepackage{amsmath}
\usepackage{xcolor}
\theoremstyle{plain}
\newtheorem{theorem}{Theorem}[section]
\newtheorem{lemma}[theorem]{Lemma}

\newtheorem{proposition}[theorem]{Proposition}
\theoremstyle{definition}
\newtheorem{remark}[theorem]{Remark}

\def\Rn{\mathbb R\sp n}

\def\nt{\|^{^{\oslash}}}

\def\R{\mathbb R}
\def\N{\mathbb N}
\def\M{\mathcal M}

\def\oX{\overline{X}}
\def\G{\mathcal G}
\def\C{\mathcal C}

\newtoks\by
\newtoks\paper
\newtoks\book
\newtoks\jour
\newtoks\yr
\newtoks\pages
\newtoks\vol
\newtoks\publ
\def\ota{{\hbox\vol{???}}}
\def\cLear{\by=\ota\paper=\ota\book=\ota\jour=\ota\yr=\ota
\pages=\ota\vol=\ota\publ=\ota}
\def\endpaper{\the\by, \the\paper.
{\it\the\jour\/} {\bf \the\vol} (\the\yr), \the\pages.\cLear}
\def\endbook{\the\by, {\it\the\book}. \the\publ.\cLear}
\def\endprep{\the\by, \the\paper. \the\jour.\cLear}

\numberwithin{equation}{section}

\def\loc{\operatorname{loc}}

\begin{document}

\title{Norms supporting  the Lebesgue differentiation theorem}

\begin{abstract} A version of the Lebesgue differentiation theorem is offered, where the $L^p$ norm is replaced with
any
 rearrangement-invariant norm. Necessary and sufficient conditions for a norm of this kind to support the Lebesgue
  differentiation theorem are established. In particular,  Lorentz,
   Orlicz and other customary norms for which Lebesgue's theorem holds are characterized.
\end{abstract}

\author{Paola Cavaliere, Andrea Cianchi,  Lubo\v s Pick and Lenka Slav\'{\i}kov\'a}

\address{Dipartimento di Matematica\\ Universit\`a di Salerno\\ Via Giovanni Paolo II, 84084 Fisciano (SA), Italy}
 \email{pcavaliere@unisa.it}

\address{Dipartimento di Matematica e Informatica \lq \lq U. Dini''\\
Universit\`a di Firenze\\  Viale Morgagni 57/A, 50134 Firenze,
Italy}
 \email{cianchi@unifi.it}

\address{Department of Mathematical Analysis
\\  Faculty of Mathematics and Physics,
Charles University
\\  Sokolovsk\'a~83,
186~75 Praha~8,
Czech Republic}
\email{pick@karlin.mff.cuni.cz}

\address{Department of Mathematical Analysis
\\  Faculty of Mathematics and Physics,
Charles University
\\  Sokolovsk\'a~83,
186~75 Praha~8,
Czech Republic}
\email{slavikova@karlin.mff.cuni.cz}


\thanks{This research was partly supported by
the  research project of MIUR
(Italian Ministry of Education, University and Research) Prin 2012,
n. 2012TC7588,  ``Elliptic and parabolic partial differential
equations: geometric aspects, related inequalities, and
applications",  by GNAMPA of the Italian INdAM (National
Institute of High Mathematics, by
 the grant P201/13/14743S of the Grant Agency of
the Czech Republic, and by the research project GA UK No. 62315 of
the Charles University in Prague.}
\maketitle

\section{Introduction and main results}\label{Intro}
A standard
 formulation of the classical Lebesgue differentiation theorem   asserts
that, if $u \in L^1_{\loc}(\Rn)$, $n\geq 1$, then
\begin{equation}\label{averages}\displaystyle
\lim _{r \to 0^+} \frac{1}{{\mathcal L}^n (B_{r}(x))}\int
_{B_{r}(x)} u(y)  \, d{\mathcal{\mathcal L}^n }(y)\qquad
\hbox{exists and is finite for  a.e. $x \in \Rn$,}
\end{equation}
 where ${\mathcal L}^n$ denotes the Lebesgue measure in $\Rn$,  and $B_r(x)$      the ball,  centered at $x$,  with radius $r$.  Here, and in what follows, \lq\lq a.e." means \lq\lq almost every" with respect to Lebesgue   measure.
In addition to \eqref{averages}, one has that
\begin{equation}\label{leb in L1}\displaystyle
\lim _{r \to 0^+}\    \| u - u(x) \nt_{L ^1(B_r(x))} = 0\qquad
\hbox{for   a.e. $x \in \Rn$},
\end{equation}
where $\| \cdot  \nt_{L ^1(B_r(x))}$ stands for the averaged norm in
$L^1(B_r(x))$ with respect to the normalized Lebesgue measure $
\tfrac{1}{{\mathcal L}^n (B_{r}(x))}\mathcal L^n$. Namely,
$$
\|u\nt_{{L^{1}(B_r(x))}}=
\tfrac{1}{{\mathcal L}^n (B_{r}(x))}\int _{B_{r}(x)} |u(y)|  \,
d{\mathcal{\mathcal L}^n }(y)
$$
for $u \in {L^1_{\loc}(\Rn)}$.
\par
A slight extension of this property ensures
that an analogous conclusion holds if the $L ^1$-norm in \eqref{leb in L1} is replaced with any
$L^p$-norm, with $p \in [1, \infty)$. Indeed, if $u \in
L^p_{\loc}(\Rn)$, then
\begin{equation}\label{leb in Lp}\displaystyle
\lim _{r \to 0^+}\    \| u - u(x) \nt_{L ^p(B_r(x))} = 0\qquad
\hbox{for   a.e. $x \in \Rn$},
\end{equation}
the averaged norm $\| \cdot \nt_{L^p(B_r(x))}$ being defined
accordingly. By contrast,   property \eqref{leb in Lp}
fails when $p=\infty$.
\par
The question thus arises of a characterization of those norms,
defined on the space $L^0(\Rn)$   of measurable functions on $\Rn$,
for which a version of the Lebesgue differentiation theorem
continues to hold.

In the present paper we address this issue in the class of all
rearrangement-invariant   norms, i.e.
  norms which only depend on the \lq\lq size"
of functions, or, more precisely, on the measure of their level
sets. A precise definition of this class of norms, as well as other
notions employed hereafter, can be found in
Section~\ref{S:background} below, where the necessary background
material is collected.

Let us just recall here that, if   $\| \cdot\|_{X(\Rn)}$ is a
rearrangement-invariant norm, then
\begin{equation}\label{r.i.} \|u \|_{X(\Rn)}=\| v\|_{X(\Rn)}  \qquad
{\hbox{whenever $u^*=v^*$}},  \end{equation}
 where $u^*$ and $v^*$ denote the decreasing rearrangements of the functions
 $u, v \in
L^0(\Rn)$. Moreover, given any
  norm of this kind, there exists another  rearrangement-invariant
function norm
\mbox{$\| \cdot \|_{\overline X(0,\infty)}$}~on
$L^0(0,\infty)$, called the representation norm of  {$\|\cdot\|_{X(\Rn)}$,}  such that
\begin{equation}\label{repr}
 \| u\|_{X(\Rn)} = \| u^*\|_{\overline X(0,\infty)}
\end{equation}
for every $u\in L^0(\Rn)$.
By $X(\Rn)$ we denote  the Banach function space, in the sense of
Luxemburg,  of all functions $u \in L^0(\Rn)$ such that $\|
u\|_{X(\Rn)}< \infty$.
Classical instances of rearrangement-invariant function norms   are
Lebesgue, Lorentz, Orlicz, and Marcinkiewicz
norms.
\par
In analogy with \eqref{leb in Lp}, a rearrangement-invariant norm
$\| \cdot\|_{X(\Rn)}$   will be said to satisfy the Lebesgue point
property if, for every $u \in X_{\loc}(\Rn)$,
\begin{equation}\label{leb in X2}\displaystyle
\lim _{r \to 0^+}\    \| u - u(x) \nt_{X(B_r(x))} = 0\qquad
\hbox{for   a.e. $x \in \Rn$}.
\end{equation}
Here, $\| \cdot \nt_{X(B_r(x))}$ denotes the   norm  on
 $X(B_r(x))$  with respect to the normalized Lebesgue
measure $ \tfrac{1}{{\mathcal L}^n (B_{r}(x))}\mathcal L^n$ -- see
\eqref{normaaverage'}, Section~\ref{S:background}.
\par
We shall exhibit necessary and sufficient conditions for $\|
\cdot\|_{X(\Rn)}$ to enjoy the Lebesgue point property. To begin
with, a necessary condition for $\| \cdot\|_{X(\Rn)}$  to satisfy
the Lebesgue point property is to be locally absolutely continuous
(Proposition~\ref{P:necessary}, Section~\ref{localabs}). This
means that,  for each  function $u \in
X_{\loc}(\Rn)$, one has $\displaystyle \lim_{j \to \infty} \|u\chi_{{K_j}}\|_{{X (\Rn)}}=0$
for every non-increasing sequence {$\{K_j\}$} of
bounded measurable sets  in $\Rn$ such that $\cap_{j\in \N} K_j =
\emptyset$.
\\
The local absolute continuity of  $\| \cdot\|_{X(\Rn)}$ is in turn
equivalent to the local separability of $X(\Rn)$, namely
 to  the separability of each
subspace of $X(\Rn)$ consisting of all functions which are supported
in any given bounded measurable subset of $\Rn$.
\par
As  will be clear from applications of our results to special
instances, this  necessary assumption is not yet sufficient. In
order to ensure the Lebesgue point property for $\|
\cdot\|_{X(\Rn)}$, it has to be complemented with
an additional assumption on
 the functional $\mathcal G _X$, associated with the representation norm $\|\cdot \|_{\overline X(0,\infty)}$, and defined as
 \begin{equation}\label{functional}
\mathcal G _X (f) = \|f^{-1}\|_{\overline X(0,\infty)}
\end{equation}
for every non-increasing function $f: [0,\infty) \rightarrow
[0,\infty]$. Here, $f^{-1} : [0,\infty) \rightarrow [0, \infty]$
denotes the (generalized) right-continuous inverse of $f$. Such an
assumption amounts to requiring  that  $\mathcal G_X$ be  \lq\lq
almost concave". By this expression,   we mean that the functional
$\mathcal G _X$, restricted to the convex set $\mathcal C$ of all
non-increasing functions from $[0,\infty)$ into $[0,1]$, fulfils the
inequality in the definition of concavity possibly
 up to a multiplicative positive constant $c$.
Namely,
\begin{equation}\label{nearly}
c \sum _{i=1}^k \lambda _i \mathcal G _X (f_i) \leq \mathcal G _X
\bigg(\sum _{i=1}^k \lambda _i f_i\bigg)
\end{equation}
%
%
 for any numbers
$\lambda _i \in (0, 1)$, $i=1, \dots, k$, $k\in\N$, such that
$\sum_{i=1}\sp k\lambda_i=1$, and any functions $f _i\in \mathcal
C$, $i=1, \dots, k$.
 Clearly, the functional $\mathcal G _X$ is concave on $\mathcal C$,
in the usual sense, if inequality \eqref{nearly} holds with $c=1$.

\begin{theorem}\label{T:main-theorem1}
A rearrangement-invariant norm
$\|\cdot\|_{{X(\Rn)}}$ satisfies  the Lebesgue point property if, and
only if, it is locally absolutely continuous and  the functional $\mathcal G _X$ is almost concave.
\end{theorem}

\begin{remark}\label{rem1}
{\rm In order to give an idea of how the functional $\mathcal G _X$
looks like in classical instances, consider the case when
 $\|\cdot\|_{X(\Rn)}= \|\cdot\|_{L^p(\Rn)} $.
 One has that
$$
\G_{L^p}(f)=\begin{cases}\left(p\int_0^\infty s^{p-1} f(s)\,d \mathcal L^1
(s)\right)^{\frac 1p} \qquad & \mbox{if} \  p \in [1,\infty),\\
 \mathcal L^1(\{s\in [0,\infty): f(s)>0\}) & \mbox{if} \  p =\infty\,, \end{cases}
$$
for every non-increasing function $f: [0,\infty) \rightarrow
[0,\infty]$}. The functional $\G_{L^p}$ is concave for every $p \in
[1, \infty]$. However, $\|\cdot\|_{L^p(\Rn)}$ is locally absolutely
continuous only for $p<\infty$.
\end{remark}

\begin{remark}\label{rem1'}
The local absolute  continuity of a rearrangement invariant norm
$\|\cdot\|_{{X(\Rn)}}$  and the  almost concavity
of the functional $\mathcal G_X$ are independent properties.
 For instance, as noticed in the previous
remark, the norm $\|\cdot\|_{L^\infty(\Rn)}$ is not locally
absolutely continuous,  although  the functional $\mathcal
G_{L^\infty}$ is concave. On the
other hand,   whenever $q < \infty$, the Lorentz norm $\|\cdot\|_{L^{p,q}(\Rn)}$ is locally
absolutely continuous  , but  $\mathcal
G_{L^{p,q}}$  is    almost  concave if and only if $q\leq p$. The
Luxemburg norm $\|\cdot\|_{L^{A}(\Rn)}$ in the Orlicz space
$L^{A}(\Rn)$ is almost concave for every $N$-function $A$, but is
locally absolutely continuous if and only if $A$ satisfies the
$\Delta _2$-condition near infinity. These properties are
established in Section \ref{S:examples} below, where the validity of
the Lebesgue point property for various classes of norms is
discussed.
\end{remark}

An alternative characterization of the rearrangement-invariant norms
satisfying the Lebesgue point property involves
 a maximal function operator
associated with the norms in question. The relevant operator,
denoted by $\M_{{X}}$, is defined,  at each $u \in
X_{\loc}(\Rn)$, as
\begin{equation}\label{HL}
\M_{{X}} u(x)=\sup_{B\ni x} \|u\nt_{{X(B)}}
\qquad \hbox{for $x\in \Rn$},
\end{equation}
 where $B$ stands for any ball in $\Rn$.

 In the case  when ${X(\Rn)}={L^1(\Rn)}$, the
operator $\M_{X}$ coincides with the classical Hardy-Littlewood
maximal operator $\M$. It is well known that $\M$ is of weak type
from $L^1(\Rn)$ into $L^1(\Rn)$. Moreover,
since
$$\|u^*\nt_{L^1(0,s)} = \frac{1}{s}\int_0^s u^*(t)\,d \mathcal L^1 (t)  \qquad\text{for}\ s
\in  (0, \infty),$$ for every $u\in L^1_{\loc}(\Rn)$, the celebrated
Riesz-Wiener inequality takes the form
%
$$
(\M u)^*(s)\leq  C \|u^*\nt_{L^1(0,s)} \qquad\text{for}\ s
\in  (0, \infty),
$$
for  some constant $C=C(n)$ \cite[Theorem 3.8, Chapter 3]{BS}.
\par
The validity of the Lebesgue point property for a
rearrangement-invariant norm $\| \cdot \|_{X(\Rn)}$ turns out to be
intimately connected to a suitable version of these two results for
the maximal operator $\M_{X}$  defined by
\eqref{HL}.  This is the content of our next result, whose statement
makes use of a notion of weak-type operators between local
rearrangement-invariant spaces. We say that $\M_{X} $ is of weak
type from $X_{\loc}(\Rn)$ into $L^1_{\loc}(\Rn)$  if for every
bounded measurable set $K\subseteq \Rn$,  there exists a constant
$C=C(K)$ such that
\begin{equation}\label{w-estHL}\mathcal L^n(\{x \in K: \M_{X} u(x)>t\}) \leq \frac Ct \|u\|_{X(\Rn)}  \qquad\text{for}\ t
\in  (0, \infty),
\end{equation}
 for every function $u\in
X_{\loc}(\Rn)$ whose support is contained in $K$.

\begin{theorem}\label{T:main-theorem2}
Let
$\|\cdot\|_{{X(\Rn)}}$ be a rearrangement-invariant norm. Then the
following statements are equivalent:
\par\noindent
\textup{\ \ (i)} $\|\cdot\|_{X(\Rn)}$ satisfies the Lebesgue point
property;
\par\noindent
\textup{\ (ii)} $\|\cdot\|_{X(\Rn)}$ is locally absolutely continuous, and
 the  Riesz-Wiener type inequality
\begin{equation}\label{E:riesz}
(\M_{X}u)^*(s) \leq C \|u^*\nt_{\overline{X}(0,s)}  \qquad\text{for}\ s
\in  (0, \infty),
\end{equation}
holds for some positive
constant $C$, and for every $u\in X_{\loc}(\Rn)$;
\par\noindent
\textup{(iii)}  $\|\cdot\|_{X(\Rn)}$ is locally absolutely continuous, and the operator $\M_{{X}}$
is of weak type from $X_{\loc}(\Rn)$ into $L\sp1_{\loc}(\Rn)$.
\end{theorem}

\begin{remark}\label{rem2}
 {\rm The  local absolute continuity of the norm $\|\cdot\|_{X(\Rn)}$  is  an indispensable hypothesis in both conditions (ii)
  and (iii)  of  Theorem~\ref{T:main-theorem2}.  Indeed, its necessity is already known from Theorem~\ref{T:main-theorem1},
  and, on the other hand, it does not follow  from the other assumptions in (ii) or (iii).  For instance, both  these
    assumptions are fulfilled by the rearrangement-invariant
norm   $\|\cdot\|_{L^\infty(\Rn)}$, which, however, is not locally
absolutely continuous, and, in fact, does not satisfy the Lebesgue
point property. }
\end{remark}

\begin{remark}\label{rem2'}
 {\rm  Riesz-Wiener type inequalities for special classes of
rearrangement-invariant norms have been  investigated in the
literature -- see e.g. \cite{BMR, BP,  Lec, Le}. In particular,
in \cite{BP}
 inequality \eqref{E:riesz} is shown to hold  when $\|\cdot
 \|_{X(\Rn)}$ is an Orlicz norm $\|\cdot\|_{L^A(\Rn)}$ associated with any Young function $A$. The case
of  Lorentz norms $\|\cdot\|_{L^{p,q}(\Rn)}$ is treated in
\cite{BMR}, where it is proved that \eqref{E:riesz} holds    if, and
only if, $1\leq q\leq p$.
In fact,  a different notion of maximal operator is considered in
\cite{BMR}, which, however, is equivalent to \eqref{HL} when
$\|\cdot \|_{X(\Rn)}$ is a Lorentz norm, as is easily seen from \cite[Equation $(3.7)$]{CC}.

\noindent  A simple sufficient condition for the validity  of the
Riesz-Wiener type inequality for very general maximal operators is
proposed in \cite{Le}. In   our framework, where maximal operators
built upon rearrangement-invariant norms are taken into account,
such condition turns out to be also necessary, as will be shown in
Proposition~\ref{P:herz}. The approach introduced in~\cite{Le} leads
to alternative proofs of the Riesz-Wiener type inequality for Orlicz
and Lorentz norms, and was also used in~\cite{MP} to prove the
validity of~\eqref{E:riesz} for further families of
rearrangement-invariant norms, including, in particular, all Lorentz
endpoint norms $\|\cdot\|_{\Lambda_\varphi(\Rn)}$. A kind of
rearrangement inequality for the maximal operator built upon these
Lorentz  norms already appears  in~\cite{Lec}.
\\
 Results on weak type boundedness of
the maximal operator $\mathcal M_X$ are available in the literature
as well \cite{Ai, CK, L, Pe, Stein}. For instance,  in \cite{Stein} it is pointed out that the
operator $\mathcal M_{L^{p,q}}$ is of weak type from
$L^{p,q}(\Rn)$ into $L^p(\Rn)$, if $1\leq q\leq p$,
and hence, in particular, it is of weak type from
$L^{p,q}_{\loc}(\Rn)$ into $L^1_{\loc}(\Rn)$.}
\end{remark}

Our last main result provides us with  necessary and sufficient
conditions for the Lebesgue point property of a
rearrangement-invariant  norm which do not make explicit reference
to the local absolute continuity of the relevant norm.

\begin{theorem}\label{T:main-theorem3}
Let $\|\cdot\|_{{X(\Rn)}}$ be a rearrangement-invariant norm. Then
the following statements are equivalent:
\par\noindent
\textup{\ \ (i)} $\|\cdot\|_{X(\Rn)}$ satisfies the Lebesgue point
property;
\par\noindent
\textup{\ (ii)} For every function $u\in X(\Rn)$,
supported in a set of  finite  measure,
\[
\mathcal L^n (\{x\in\Rn\colon \M_{X}u(x)>1\})<\infty; \]
\par\noindent
\textup{(iii)}  For every function $u\in X(\Rn)$,
supported in a set of  finite  measure,
 $$\displaystyle \lim_{s\to \infty}
(\M_{X}u)^*(s)=0.$$
\end{theorem}

Theorems~\ref{T:main-theorem1},~\ref{T:main-theorem2} and~\ref{T:main-theorem3}  enable us to characterize the validity of the
Lebesgue point property in customary classes of
rearrangement-invariant norms.

The following proposition deals with
the case of standard Lorentz norms  $\|\cdot\|_{{L^{p,q}(\Rn)}}$.

\begin{proposition}\label{T:Lorentz} The Lorentz  norm $\|\cdot\|_{{L^{p,q}(\Rn)}}$
satisfies  the Lebesgue point property if, and only if, $1\leq q
\leq p < \infty$.
\end{proposition}

Since $L^{p,p}(\Rn) = L^p(\Rn)$, Proposition~\ref{T:Lorentz}
recovers, in particular, the standard result, mentioned above,  that
the    norm $\|\cdot\|_{{L^{p}(\Rn)}}$ enjoys the Lebesgue
point property if, and only if, $1 \leq p < \infty$.
\\ This fact is also reproduced by the following proposition,
which concerns Orlicz norms $\|\cdot\|_{{L^{A}(\Rn)}}$ built upon a
Young function $A$.

\begin{proposition}\label{T:Orlicz} The Orlicz   norm
 $\|\cdot\|_{{L^{A}(\Rn)}}$  satisfies  the Lebesgue point property if, and only
if, the Young function $A$ satisfies the $\Delta_2$-condition near
infinity.
\end{proposition}

The last two results concern the so called Lorentz and Marcinkiewicz
endpoint norms \, \mbox{$\|\cdot\|_{{\Lambda_\varphi(\Rn)}}$}~and
$\|\cdot\|_{M_\varphi(\Rn)}$, respectively, associated with  a (non
identically vanishing)
 concave  function $\varphi : [0, \infty ) \to [0, \infty)$.

\begin{proposition}\label{T:Lambda} The Lorentz  norm $\|\cdot\|_{{\Lambda_{\varphi}(\Rn)}}$
satisfies  the Lebesgue point property if, and only if,
$\displaystyle \lim_{s\to 0^+}\varphi(s)=0$.
\end{proposition}

\begin{proposition}\label{T:Marcinkiewicz} The  Marcinkiewicz  norm $\|\cdot\|_{M_{\varphi}(\Rn)}$
satisfies  the Lebesgue point property if, and only if,
$\displaystyle \lim_{s\to 0^+}\tfrac s{\varphi(s)}>0$, namely, if and only if,
$(M_{\varphi})_{\loc}(\Rn)= L^1_{\loc}(\Rn)$.
\end{proposition}

 When the present paper was almost in
 final form, it was pointed out to us by A.~Gogatishvili that the
Lebesgue point property of rearrangement-invariant spaces has also
been investigated in \cite{BS1, BS2, P, SS}. The analysis of those
papers  is however limited to the case of functions of one variable.
Moreover, the characterizations of
those norms having Lebesgue point property that are proved  there
are less explicit, and have a somewhat more technical nature.

\section{Background}\label{S:background}

In this section we recall some definitions and basic properties  of
decreasing rearrangements and rearrangement-invariant function
norms. For more details and proofs,  we refer to \cite{BS, PK}.

Let $E$ be a Lebesgue-measurable subset of $\Rn$, $n \geq 1$. The
Riesz space of
 measurable functions from $E$ into $[-\infty , \infty]$ is denoted by $L^0(E)$. We also set
$L^0_+(E)= \{u \in L^0(E) : u \geq 0 \ \hbox{a.e. in} \, E\}$, and
$L^0_0(E)= \{u \in L^0(E) : u \,\hbox{is finite}\, \hbox{a.e. in} \,
E\}$.  The \emph{distribution function} $u_* : [0, \infty)
\rightarrow [0, \infty]$ and the \emph{decreasing rearrangement}
$u^*: [0, \infty) \rightarrow [0, \infty]$ of a function $u\in
L^0(E)$  are defined by
\begin{equation}\label{distributionE}
u_*(t)= {\mathcal L^n}(\{y \in E:  |u(y)|>t\}) \qquad\text{for}\ t
\in  [0, \infty),
\end{equation}
and   by
\begin{equation}\label{decreasing rearrangementE}
u^*(s)= \inf\{t \geq 0:   u_*(t) \leq s\}  \qquad\text{for}\ s \in
[0, \infty),
\end{equation}
respectively.

\noindent The   Hardy-Littlewood inequality
tells us that
\begin{equation}\label{HL.0}
\int _E |u(y) v(y)|\, d\mathcal L^n(y) \leq \, \int_{0}^{ \infty
}u^{\ast}(s)v^{\ast}(s)\, d\mathcal L^1(s)  \,
\end{equation}
for every $u,v \in L^0(E)$. The  function $u^{**} :(0, \infty) \to
[0, \infty ]$, given by
\begin{equation}\label{u^**}
u^{\ast \ast}(s)\, = \, {\frac 1 s}\, \int_0^{s}u^*(t) d \mathcal
L^1 (t) \, \qquad \hbox{for} \ s \in \, (0,  \infty) \,,
\end{equation}
is non-increasing and satisfies  $u^* \leq u^{**}$. Moreover,
\begin{equation}\label{**}
(u+v)^{\ast \ast}  \, \leq \, u^{\ast \ast}  + v^{\ast \ast}
\end{equation}
for every $u, v \in L^0_+ (E)$.

A \emph{rearrangement-invariant norm}  is a functional $\|
\cdot\|_{{X(E)}}: {L^0(E)}
\rightarrow [0, \infty]$ such that
\begin{description}
\item[\rm (N1)] {$\|u+v\|_{{X(E)}} \, \leq \, \|u\|_{{X(E)}} + \|v\|_{{X(E)}}   \quad  \text{for all}\  u,v \in L^0_+(E)$;

\ $\|\lambda u\|_{{X(E)}} \,  = \,  |\lambda| \, \|u\|_{{X(E)}} \quad   \text{for all} \
{\lambda \in \R ,\,  u \in L^0(E)}$;

\ $\|u\|_{{X(E)}} \, >\, 0  \quad  \text{if} \  u   \, \text{does
not vanish a.e. in $E$;} $}

\item[\rm (N2)]   $\|u\|_{{X(E)}} \, \leq \, \|v\|_{{X(E)}}$   whenever    $0 \leq u \, \leq \, v$  a.e. in $E$;

\item[\rm (N3)]  $\sup_{_k}\|u_k\|_{{X(E)}} \, = \, \| u\|_{{X(E)}} $    if   $\{u_k\} \subset L^0_{+}(E)$   with    $u_k \nearrow u$   a.e. in $E$;

\item[\rm (N4)]  $\|\chi_{G}\|_{{X(E)}} \, < \, \infty$ for every measurable set $G\subseteq E$, such that $\mathcal L^n(G)< \infty$;

\item[\rm (N5)]   for every measurable set $G\subseteq E$, with $\mathcal L^n(G)< \infty$,  there exists a positive constant $C(G)$ such that
$\|u\|_{{L^1(G)}}  \leq C(G) \, \|u\chi_{G}\|_{{X(E)}}   \ \text{for all}\
u  \in {L^0(E)}$;

\item[\rm (N6)] $\|u\|_{{X(E)}} \, = \, \| v\|_{{X(E)}}$  for all $u, v \in  L^0(E)$  such that $u^\ast = v^\ast$.
\end{description}

\medskip
\par\noindent
The functional $\| \cdot\|_{{X(E)}}$ is a norm in the standard
sense when restricted to the set
\begin{equation}\label{space X1}
X(E)=\{u \in L^0(E):  \|u\|_{{X(E)}}< \infty\}.
\end{equation}
The latter is a Banach space  endowed with such norm, and is called
a rearrangement-invariant Banach function space, briefly, a
\emph{rearrangement-invariant space}.
\par
Given a measurable subset $E'$ of $E$ and a function $u \in
L^0(E')$, define the function $\widehat u \in L^0(E)$ as
$$\widehat u(x) = \begin{cases} u(x) & \hbox{if $x \in E'$},
\\ 0 & \hbox{if $x \in E\setminus E'$.}
\end{cases}$$
Then the functional $\|\cdot\|_{X(E')}$ given by
$$\|u\|_{X(E')}=\|\widehat u\|_{X(E)}$$
for $u \in L^0(E')$ is a rearrangement-invariant norm.
%
%
\par
 If $\mathcal L^n (E) < \infty$, then
\begin{equation}\label{imm}
L^\infty (E)     \to  X(E)     \to L^1(E),
\end{equation} where $\to$ stands for a continuous embedding.
\\
The local r.i. space $X_{\text{loc}}(E)$ is defined as
\begin{equation}\nonumber
X_{\rm loc}(E)   \ = \{ u \in \ L^0(E): \  u \chi_K \in X (E) \,  \
\text{for every bounded measurable set} \,   K \subset E
 \}.
\end{equation}
 The {\it fundamental function} of $
X(E)$  is defined by
\begin{equation}\label{f.f.}\varphi_{{X(E)}}(s)=  \, \| \chi_{G}\|_{{X(E)}}  \, \qquad
\hbox{for} \ s \in [0,\mathcal L ^n (E)),
\end{equation}  where $G$ is any measurable  subset of $E$ such that $\mathcal L^n(G)=s$.
 It is
non-decreasing on $[0,\mathcal L ^n (E))$, $\varphi_{X(E)}(0)=0$ and
$\varphi_{X(E)}(s)/s$ is non-increasing for $s \in \, (0,\mathcal L ^n
(E))$.
\par
\noindent Hardy's Lemma tells us that, given  $u, v \in L^0(E)$ and
any rearrangement-invariant norm $\| \cdot \|_{{X(E)}}$,
\begin{equation} \label{Hardy}
\hbox{if}\,\, u^{**}  \leq v^{**}, \,\, \hbox{then} \,\, \|u
\|_{{X(E)}} \leq \| v \|_{{X(E)}}.
 \end{equation}

The \emph{associate rearrangement-invariant norm} of $\|
\cdot\|_{{X(E)}}$ is the rearrangement-invariant norm $\|
\cdot\|_{X'(E)}$ defined by
\begin{equation} \label{n.assoc.}
\|v\|_{X'(E)}=  \sup \big\{\smallint _E |u(y)v(y)|\, d\mathcal
L^n(y):\, u \in L^0(E), \, \| u \|_{{X(E)}} \leq 1 \big\}.
\end{equation}
The corresponding rearrangement-invariant space $X'(E)$  is called
the \textit{associate space} of $X(E)$. The H\"{o}lder type
inequality
\begin{equation}\label{H}
\int _E |u(y)v(y)|\, d\mathcal L^n(y)  \leq  \| u\|_{{X(E)}}
\|v\|_{{X'(E)}}
\end{equation}
holds for every $u \in X(E)$ and $v \in X'(E)$. One has that $X(E) \
= X^{''} (E)$.
\\

The rearrangement-invariant norm, defined as
$$\| f\|_{\overline X(0,\mathcal L^n (E))} = \sup _{\|u\|_{X'(E)}\leq
1} \int _0^\infty f^*(s) u^*(s)\, d \mathcal L^1 (s) $$ for $f \in
L^0(0, \mathcal L^n (E))$, is a \emph{representation norm} for $\|
\cdot\|_{{X(E)}}$. It has the property that
\begin{equation}\label{repr2}
 \| u\|_{X(E)} = \| u^*\|_{\overline X(0,\mathcal L^n (E))}
\end{equation}
 for every $u \in X(E)$. For customary rearrangement-invariant norms, an expression for $\|
\cdot \|_{\overline X(0,\mathcal L^n (E))}$   is immediately derived
from that of $\| \cdot\|_{{X(E)}}$.

%

The \textit{dilation operator} $D_{\delta}: \overline X(0,\mathcal
L^n (E)) \to \overline X(0,\mathcal L^n (E))$ is defined for $\delta
> 0$ and $f \in \overline X(0,\mathcal L^n (E))$ as
\begin{equation}\label{dilationop}
  (D_{\delta}f)(s)=  \begin{cases}f(s \delta)  & \mbox{if} \ s \delta\in(0,\mathcal L^n (E)), \\ 0 & \mbox{otherwise},\end{cases}
\end{equation}
 and is bounded \cite[Chap.~3, Proposition~5.11]{BS}.

We shall make use of the
 subspace $\overline X_1(0,\infty)$ of $\overline X(0,\infty)$
 defined as
\begin{equation}\label{X1}
\overline X_1(0, \infty) = \{f \in \overline X(0, \infty): \, f(s)=
0 \,\, \hbox{for a.e.}\, s >1\}.
\end{equation}

Now, assume  that $E$ is a measurable positive cone in  $\Rn$ with
vertex at $0$,  namely, a measurable set which is closed under
multiplication by positive scalars. In what follows, we shall focus
 the nontrivial case when $\mathcal L^n(E)$ does not vanish, and
 hence $\mathcal L^n(E)=\infty$.
Let $\| \cdot\|_{{X(E)}}$ be a rearrangement-invariant norm, and let
$G$ be a measurable subset of $E$  such that $0<{\mathcal L}^n(G) <
\infty$.  We define the functional $\|\,\cdot\, \nt_{{X (G)}}$ as
\begin{equation}\label{normaaverage'}
\|u\nt_{{X (G)}}= \big\|(u\chi _G) (\sqrt[n]{{\mathcal L}^n(G)}  \
\cdot )\big\|_{ { X(E)}}
\end{equation}
for $u \in L^0(E)$.
We call it the {\it averaged norm of $\| \cdot\|_{{X (E)}}$ on $G$},
since
\begin{equation}\label{due
misure} \|u\nt_{{X (G)}}= \|u \chi _G\|_{{X \big(G, \frac{{\mathcal
L}^n}{{\mathcal L}^n(G)}\big)}}
\end{equation}
for $u \in L^0(E)$, where $\|\cdot \|_{{X \big(G, \frac{{\mathcal
L}^n}{{\mathcal L}^n(G)}\big)}} $ denotes the
rearrangement-invariant norm, defined as $\|\cdot\|_{X(G)}$, save
that the Lebesgue measure ${\mathcal L}^n$ is replaced with the
normalized Lebesgue measure $\frac{{\mathcal L}^n}{{{\mathcal
L}^n(G)}}$.
 Notice that
\begin{equation}\label{normaaverage}
\|u\nt_{{X (G)}}= \|(u\chi _G)^\ast  ({\mathcal L}^n(G) \, \cdot
)\|_{ {\overline X(0,\infty)}}
\end{equation}
for $u \in L^0(E)$.
Moreover,
\begin{equation}\label{norm1}
  \|1\nt_{{X(G)}} \quad \hbox{is independent of $G$.}
  \end{equation}
The H\"older type inequality for averaged norms takes the form:
\begin{equation}\label{H-in}
\frac 1{\mathcal L^n (G)} \int _G |u(y)v(y)|\, d\mathcal L^n (y)
\leq \| u\nt_{{X(G)}} \ \| v\nt_{{X'(G)}}\,
\end{equation}
for $u , v \in L^0(E)$.

We conclude this section by recalling  the definition of some
customary, and less standard, instances of rearrangement-invariant
function norms of use in our applications. In what follows, we set
$p'=\frac{p}{p-1}$ for $p \in (1, \infty)$, with the usual
modifications when $p=1$ and $p=\infty$. We also adopt the
convention that $1/\infty = 0$.
\par
 \noindent
Prototypal examples  of rearrangement-invariant function norms are
the classical Lebesgue norms. Indeed, $\|u\|_{{L^p(\mathbb R\sp n)}}=
\|u^\ast\|_{{L^p(0,\infty)}}$, if $p \in [1, \infty)$, and
$\|u\|_{_{L^\infty(\mathbb R\sp n)}}= u^\ast(0)$.

Let   $p, q \in [1, \infty]$. Assume that either $1<p<\infty$ and $1\leq
q\leq\infty$, or $p=q=1$, or $p=q=\infty$. Then the functional
defined as
\begin{equation}\label{Lorentz}
\|u\|_{{L^{p,q}(\mathbb R\sp n)}}=\| {s}^{\frac{1}{p} - \frac{1}{q}}\
u^{\ast}(s)\|_{{L^{q} (0,\infty)}}
\end{equation}
for $u \in L^0(\mathbb R\sp n)$, is equivalent (up to multiplicative
constants) to a rearrangement-invariant  norm. The corresponding
rearrangement-invariant  space is called a \textit{Lorentz space}.
 Note that $\|\cdot\|_{{L^{p,q}(0,\infty)}}$ is the representation norm for   $\|\cdot\|_{{L^{p,q}(\mathbb R\sp n)}}$, and $L\sp{p,p}(\mathbb R\sp n)=L\sp p(\mathbb R\sp n)$. Moreover,
 $L\sp{p, q}(\mathbb R\sp n) \to   L\sp{p, r}(\mathbb R\sp n)$ if  $1\leq
 q < r \leq\infty$.

 Let $A$ be a  Young function, namely a
left-continuous convex function from $[0, \infty)$ into $[0,
\infty]$, which is neither identically equal to $0$, nor to
$\infty$. The {\it Luxemburg  rearrangement-invariant  norm}
associated with $A$ is defined as
\begin{equation}\label{Orlicz}
\| u \|_{{L^{A}(\mathbb R\sp n)}} =  \ \inf \Bigg\{ \lambda > 0 \ :
\ \int_{\mathbb R\sp n} A \bigg(\frac{|u(x)|}{\lambda} \bigg)\,
d\mathcal L^n (x) \ \leq \ 1 \Bigg\}
\end{equation}
for  $u \in L^0(\mathbb R\sp n)$.
 Its representation norm is $\| u \|_{{L^{A}(0,\infty)}}$.
 The space  $L^{A}(\mathbb R\sp n)$ is
called an {\it Orlicz space}. In particular, $L^{A}(\mathbb R\sp n)=
L^{p}(\mathbb R\sp n)$ if $A(t)=t^p$ for $p \in  [1, \infty)$, and
$L^{A}(\mathbb R\sp n)= L^{\infty}(\mathbb R\sp n)$ if $A(t)=\infty
\chi_{_{(1, \infty)}}(t)$.
\\
Recall that  $A$ is said to satisfy the  \emph{$\Delta
_2$--condition near infinity} if it is finite valued  and there
exist constants $C>0$ and $t_0\geq 0$ such that
\begin{equation}\label{delta2}A(2t) \leq C A(t) \, \qquad
\text{for} \ t \in [t_0, \infty)\,.
\end{equation}
If $A$  satisfies the  $\Delta _2$--condition near infinity, and $u
\in L^{A}(\mathbb R\sp n)$ has support of finite measure, then
$$\int _{\mathbb R^n}A(c|u(x)|)d\mathcal L^n (x) < \infty$$
for every positive number $c$.
\\
A subclass of Young functions which is often considered in the
literature is that of the so called $N$-functions. A Young function
$A$ is said to be an $N$-function if it is finite-valued, and
$$\lim _{t\to 0^+}\frac {A(t)}t =0\, \qquad \lim _{t\to \infty}\frac {A(t)}t
=\infty \,.$$
\par
Let $\varphi : [0, \infty) \to [0, \infty)$ be a  concave function
 which does not vanish identically.
  Hence, in particular, $\varphi$ is non-decreasing, and $\varphi(t)>0$ for $t\in(0,\infty)$.  The {\it Marcinkiewicz}
 and {\it Lorentz endpoint norm}
associated with $\varphi$ are defined as
\begin{equation}\label{Marcinkiewicz}
\| u \|_{{M_{\varphi}(\mathbb R\sp n)}} =  \ \sup_{s \in (0,\infty)}
u^{\ast \ast}(s) \varphi(s),
\end{equation}
\begin{equation}\label{Lorentz w}
\| u \|_{{\Lambda_{\varphi}(\mathbb R\sp n)}} =  \ \int_0^{\infty}   u^{\ast}(s) d \varphi(s) ,
\end{equation}
for  $u \in L^0 (\mathbb R\sp n)$, respectively.
 The representation norms are $\|\cdot  \|_{{M_{\varphi}(0,\infty)}}$ and $\| \cdot \|_{{\Lambda_{\varphi}(0,\infty)}}$, respectively.
 The spaces
$M_{\varphi}(\mathbb R\sp n)$ and $\Lambda_{\varphi}(\mathbb R\sp n)$ are
called
  {\it Marcinkiewicz endpoint space} and  {\it Lorentz endpoint space} associated with $\varphi$.  The   fundamental functions of $
M_{\varphi}(\mathbb R\sp n)$  and $\Lambda_{\varphi}(\mathbb R\sp
n)$ coincide with $\varphi$. In fact,   $ M_{\varphi}(\mathbb R\sp
n)$  and $\Lambda_{\varphi}(\mathbb R\sp n)$ are respectively the
largest and the smallest rearrangement-invariant space whose
fundamental function is $\varphi$, and this accounts for the
expression \lq\lq endpoint" which is usually attached to their
names.
 Note the alternative expression
\begin{equation}\label{Lorentzbis}
\| f \|_{\Lambda_{\varphi}(0,\infty)} =  f^*(0) \varphi
(0^+) +  \int_0^{\infty} f^{\ast}(s)  \varphi '(s) d\mathcal L\sp 1 (s),
\end{equation}
for  $f \in L^0 (0,\infty)$, where $\displaystyle \varphi (0^+) = \lim _{s \to 0^+}
\varphi (s)$.

\section{A necessary condition: local absolute continuity}\label{localabs}

In the present section we are mainly concerned with a proof of the following necessary conditions for a rearrangement-invariant norm to satisfy the Lebesgue point property.

\begin{proposition}\label{P:necessary}
If $\|\cdot\|_{{X(\Rn)}}$ is  a~rearrangement-invariant norm
satisfying the Lebesgue point property, then:

\begin{enumerate}
\item[\ {\rm(i)}] $\|\cdot\|_{{X(\Rn)}}$ is locally absolutely continuous;
\par\noindent
\item[{\rm(ii)}] $X(\Rn)$ is locally separable.
\end{enumerate}
\end{proposition}

 The proof of Proposition~\ref{P:necessary} is split
in two steps, which are the content of the next two lemmas.

\begin{lemma}\label{L:absolute-continuity}

Let $\|\cdot\|_{{X(\Rn)}}$ be a    rearrangement-invariant norm
which satisfies the Lebesgue point property. Then:
\begin{enumerate}
\item[\textup{(H)}] Given any function $f\in  \overline{X}_1(0,\infty)$, any sequence $\{I_k\}$ of pairwise disjoint intervals in $(0,1)$,
 and any sequence $\{a_k\}$ of positive numbers such that $a_k\geq \mathcal L^1(I_{k})$ and
\begin{equation}\label{E:ak}
\|(f\chi_{I_k})^*\nt_{{ \overline{X}(0,a_k)}}>1,
\end{equation}
one has that $$\sum_{k=1}^\infty a_k<\infty.$$
\end{enumerate}
\end{lemma}

Let us  stress  in advance that condition (H) is not only necessary,
but also sufficient for a rearrangement-invariant norm to satisfy
the Lebesgue point property. This is a consequence  of
Proposition~\ref{L:v}, Section~\ref{proofsmain}, and of the
following lemma.

\begin{lemma}\label{L:absolute-continuity-bis}
If     a  rearrangement-invariant
norm $\|\cdot\|_{{X(\Rn)}}$ fulfills condition (H) of Lemma~\ref{L:absolute-continuity},
then $\|\cdot\|_{{X(\Rn)}}$ is locally absolutely continuous.
\end{lemma}

The proof of Lemma~\ref{L:absolute-continuity} in turn exploits the
following property, which will also be of later use.

\begin{lemma}\label{T:lemma_conc}
Let $\|\cdot\|_{X(\Rn)}$ be a  rearrangement-invariant  norm. Given
any function $f\in \overline{X}(0,\infty)$, the function $F : (0,
\infty) \to [0, \infty)$, defined as
\begin{equation}\label{F} \quad F(r)= r \, \|f^*\nt_{\overline{X}(0,r)} \qquad \hbox{for
$ r \in (0, \infty)$,}
\end{equation}
is non-decreasing on $(0,\infty)$,  and the function
$\frac{F(r)}{r}$ is non-increasing on $(0,\infty)$. In particular,
the function $F$ is continuous on $(0, \infty)$.
\end{lemma}
\begin{proof}
Let $0<r_1< r_2$. An application of \eqref{n.assoc.} tells us that
\begin{align*}
F(r_1)
&= \ r_1  \|(f^*\chi_{(0,r_1)})(r_1\, \cdot)\|_{{\oX(0,\infty)}}
=\ r_1 \sup_{\|g\|_{{{\oX}'(0,\infty)}}\leq 1} \ \int_0^1 g^*(s) f^*(r_1s)\,d \mathcal L^1 (s)\\
&=\ \sup_{\|g\|_{{\oX'(0,\infty)}}\leq 1} \ \int_0^{r_1}
g^*\left(\tfrac{t}{r_1}\right) f^*(t)\,d \mathcal L^1 (t) \
\leq \  \sup_{\|g\|_{{\oX'(0,\infty)}}\leq 1}\  \int_0^{r_2} g^*\left(\tfrac{t}{r_1}\right) f^*(t)\,d \mathcal L^1 (t)\\
&\leq \ \sup_{\|g\|_{{\oX'(0,\infty)}}\leq 1} \ \int_0^{r_2}
g^*\left(\tfrac{t}{r_2}\right) f^*(t)\,d \mathcal L^1 (t) \
=\ r_2 \sup_{\|g\|_{{\oX'(0,\infty)}}\leq 1} \ \int_0^1 g^*(s) f^*(r_2s)\,d \mathcal L^1 (s)\\
&=\ r_2 \|(f^*\chi_{(0,r_2)})(r_2 \, \cdot)\|_{{\oX(0,\infty)}}
=F(r_2).
\end{align*}
Namely, $F$  is non-decreasing on $(0,\infty)$. The fact  that the
function  $\frac{F(r)}{r}$ is non-increasing on $(0,\infty)$  is a
consequence of property (N2) and of the inequality
$$f^*(r_1 \, \cdot)\chi _{(0, r_1)}(r_1 \, \cdot) \geq f^*(r_2\, \cdot)\chi _{(0, r_2)}(r_2 \, \cdot)$$
 if $0<r_1< r_2$. Hence, in particular, the function $F$ is
continuous on $(0, \infty)$ (see e.g. \cite[Chapter 2, p. 49]{KS}).
\end{proof}

\begin{proof}[Proof of Lemma~\ref{L:absolute-continuity}]
Assume that $\|\cdot\|_{{X(\Rn)}}$ satisfies  the Lebesgue point
property. Suppose, by contradiction, that condition \textup{(H)}
fails, namely, there  exist a function $f\in \overline
X_1(0,\infty)$, a sequence $\{I_k\}$ of pairwise disjoint
intervals in $(0,1)$ and a sequence $\{a_k\}$ of  positive
numbers, with $a_k \geq \mathcal L^1(I_k)$, fulfilling~\eqref{E:ak}
 and such that $\sum_{k=1}^\infty a_k=\infty$.

We may assume, without loss of generality,  that
\begin{equation}\label{ip0}\displaystyle \lim_{k \to \infty} a_k=0.
\end{equation}
Indeed, if \eqref{ip0} fails, then the sequence $\{a_k\}$ can be replaced with another sequence, enjoying the same
properties, and also \eqref{ip0}. To verify this assertion, note
that, if \eqref{ip0}  does not hold, then
 there  exist   $\varepsilon
>0$ and a subsequence $\{a_{k_j}\}$ of   $\{a_k\}$ such that  $a_{k_j} \geq \varepsilon$ for all $j$.
Consider the sequence $\{b_j\}$, defined as
\begin{equation}\label{ip1}  b_j=\max \big\{\tfrac{\varepsilon}{j}, \ \mathcal L^1(I_{k_j})\big\} \quad \text{for}\ j\in \mathbb N.
\end{equation}
Equation \eqref{ip1} immediately tells us that
  $b_j\geq \mathcal L^1(I_{k_j})$, and
$\sum_{j=1}^\infty  b_j=\infty$. Moreover, Lemma~\ref{T:lemma_conc}
and  the inequality   $b_j \leq a_{k_j}$ for  $j\in \mathbb N$
ensure that  \eqref{E:ak}  holds  with $a_k$ and $I_k$ replaced by
$b_j$ and $I_{k_j}$, respectively.  Finally, $\displaystyle \lim_{j\to \infty} b_j=0$, since
$\sum_{j=1}^\infty \mathcal L^1(I_{k_j}) \leq 1$, and hence
$\displaystyle \lim_{j\to \infty}\mathcal L^1(I_{k_j})=0$.

Moreover, by skipping, if necessary, a finite number of terms in the relevant
sequences, we may also assume that
\begin{equation}\label{E:less_than_1}
\sum_{k=1}^\infty \mathcal L^1(I_{k})<1\,.
\end{equation}
Now, set $a_0=0$, and  $J_k=(\sum_{j=0}^{k-1} a_j, \sum_{j=0}^{k} a_j)$ for each $k \in \N$. We define the function  $g : (0, \infty ) \to [0,
\infty)$ as
$$
g(s)=\sum_{k=1}^\infty (f\chi_{I_k})^*\Big(s-\sum_{j=0}^{k-1}
a_j\Big)\chi_{J_{k}}(s) \qquad
\hbox{for $s\in (0, \infty)$,}
$$
and the function $u: \Rn \to [0, \infty)$ as
$$
u(y)=   \sup_{k\in \N} g(y_1+k-1) \,  \chi_{(0,1)^n}(y) \qquad \hbox{for
$y=(y_1,\dots,y_n) \in \Rn$.}
$$
The function $u$ belongs to $X(\Rn)$. To verify this fact, note that
\begin{equation}\label{E:measure}
\mathcal L^n( \{y\in \Rn: u(y) >t\})\leq \mathcal L^1 (\{s\in (0,\infty): |f \chi_{\cup_{k\in \N} I_k}|(s)>t\})
\end{equation}
for every $t\geq 0$. Indeed, thanks to the equimeasurability of  $g$ and $f\chi_{\cup_{k\in \N} I_k}$,
\begin{align*}
\mathcal L^n & ( \{y\in \Rn: u(y) >t\})
=\mathcal L^1(\{s\in (0,1): \sup_{k\in \N} g(s+k-1)>t\})\\
&=\mathcal L^1(\cup_{k\in \N} \{s\in (0,1): g(s+k-1)>t\}) \leq \sum_{k=1}^\infty \mathcal L^1(\{s\in (0,1): g(s+k-1)>t\})\\
&=\mathcal L^1(\{s\in (0,\infty): g(s)>t\}) =\mathcal L^1(\{s\in
(0,\infty): |f\chi_{\cup_{k\in \N} I_k}|(s)>t\}).
\end{align*}
From \eqref{E:measure} it follows that
$$
\|u\|_{X(\Rn)}\leq \|f\chi_{\cup_{k\in \N} I_k}\|_{\oX(0,\infty)}\leq \|f\|_{\oX(0,\infty)}<\infty\,,
$$
whence $u \in X(\Rn)$.
Next, one has  that
\begin{equation}\label{E:not_lebesgue_point}\displaystyle
\limsup_{r\to 0^+} \|u\nt_{X(B_r(x))}>0 \qquad \hbox{for  a.e. $x\in
(0,1)^n$.}
\end{equation}
To prove \eqref{E:not_lebesgue_point}, set
$$\Lambda_k=\{l\in \N: J_l\subseteq [k-1,k]\} \qquad \hbox{for}\  k\in \N.$$ Since
$\sum_{l=0}^\infty a_l=\infty$ and $\displaystyle \lim_{l\to \infty}
|J_l|=\lim_{l\to \infty} a_l=0$, we have that
\begin{equation}\label{sets}\displaystyle
\lim_{k\to \infty} \bigcup_{l\in \Lambda_k} (\overline{J_l}-k+1) = (0,1),
\end{equation}
where $\overline{J_l}$ denotes the closure of the open interval
$J_l$.
 Equation \eqref{sets} has to be interpreted in the following
set-theoretic sense: fixed any $x \in (0,1)^n$, there exist $k_0$
and an~increasing sequence $\{l_k\}_{k=k_0}^\infty$ in $\N$ such
that $l_k\in \Lambda_k$ and $x_1\in (\overline{J_{l_k}}-k+1)$ for
all $k\in \N$ greater than $k_0$. Such a
 $k_0$ can be chosen so that
$B_{\sqrt{n}a_{l_k}}(x)\subseteq (0,1)^n$ for all $k\geq k_0$, since $\displaystyle \lim_{k\to \infty} a_{l_k}=0$.

On the other hand, for every $k\geq k_0$, one also has
\begin{align*}
B_{\sqrt{n}a_{l_k}}(x)
& \ \supseteq \ \prod_{i=1}^n \, [x_i-a_{l_k}, x_i+a_{l_k}]
\ \supseteq \ (J_{l_k}-k+1) \, \times \, \prod_{i=2}^n \, [x_i-a_{l_k}, x_i+a_{l_k}].
\end{align*}
Consequently, for every $y \in \Rn$,
\begin{align*}
u\chi_{B_{\sqrt{n}a_{l_k}}(x)}(y)\
& \geq \ g(y_1+k-1) \ \chi_{(J_{l_k}-k+1)}(y_1) \,  \prod_{i=2}^n\,  \chi_{[x_i-a_{l_k}, x_i +a_{l_k}]} (y_i) \\
&= \ (f\chi_{I_{l_k}})^*\Big(y_1+k-1-\sum_{j=0}^{k-1}
a_j\Big) \ \chi_{(J_{l_k}-k+1)}(y_1) \prod_{i=2}^n \chi_{[x_i-a_{l_k}, x_i
+a_{l_k}]} (y_i)  .
\end{align*}
Hence,
\begin{equation}\label{E:rearrangement-estimate}
(u\chi_{B_{\sqrt{n}a_{l_k}}(x)})^*(s) \ \geq\
(f\chi_{I_{l_k}})^*\left( { (2 a_{l_k})^{1-n} s}\right)  \qquad
\hbox{for $s\in (0, \infty)$.}
\end{equation}
Therefore, thanks to the boundedness on rearrangement-invariant
spaces of the dilation operator, defined as in  \eqref{dilationop},
one gets
\begin{align*}
\|u\nt_{X(B_{\sqrt{n}a_{l_k}}(x))}
&=\|(u\chi_{B_{\sqrt{n}a_{l_k}}})^*(\mathcal L^n(B_{\sqrt{n}a_{l_k}}(x))\, \cdot)\|_{\oX(0,\infty)}\\
& \geq C  \, \|(f\chi_{I_{l_k}})^*(a_{l_k}\,  \cdot)\|_{\oX(0,\infty)}
=\|(f\chi_{I_{l_k}})^*\nt_{\oX(0,a_{l_k})} >1
\end{align*}
 for some positive constant $C=C(n)$.
Hence inequality~\eqref{E:not_lebesgue_point} follows, since $\displaystyle \lim_{k\to \infty} a_{l_k}=0$.

To conclude, consider the set
$M=\{y\in (0,1)^n: u(y)=0\}$. This set $M$ has positive measure. Indeed,   \eqref{E:measure} with $t=0$ and~\eqref{E:less_than_1} imply
\begin{align*}
\mathcal L^n(M)
=1-\mathcal L^n(\{y\in (0,1)^n: u(y)>0\})
& \geq 1-\mathcal L^1 (\{s\in (0,\infty): |f \chi_{\cup_{k\in \N}
I_k}|(s)>0\})\\ &\ \geq 1-\sum_{k=1}^\infty \mathcal L^1(I_{k})>0.
\end{align*}
Then, estimate \eqref{E:not_lebesgue_point} tells us that
$$
\limsup_{r\to 0^+} \|u-u(x)\nt_{X(B_r(x))}=\limsup_{r\to 0^+}
\|u\nt_{X(B_r(x))}>0\qquad \hbox{for a.e. $x\in M$.}
$$
This contradicts the Lebesgue point property for $\|\cdot\|_{X(\Rn)}$.
\end{proof}

\begin{proof}[Proof of Lemma~\ref{L:absolute-continuity-bis}]  Let
$\|\cdot\|_{{X(\Rn)}}$ be a rearrangement-invariant norm satisfying
condition (H). We first prove that, if (H) is in force, then
\begin{equation}\label{E:sufficient}
\lim_{t\to 0^+}
\|g^*\chi_{(0,t)}\|_{\overline{X}(0,\infty)}=0
\end{equation}
for every $g\in
\overline{X}_1(0,\infty)$.

\noindent Arguing
by contradiction,   assume the existence of some
$g\in \overline{X}_1(0,\infty)$ for
which~\eqref{E:sufficient} fails. From property (N2) of
rearrangement-invariant norms, this means  that some $\varepsilon>0$ exists such
that $\|g^*\chi_{(0,t)}\|_{\overline{X}(0,\infty)} \geq \varepsilon$ for every $t\in
(0,1)$. Thanks to (N1), we may assume, without loss of generality,  that
$\varepsilon=2$.
\\
Then, by  induction,    construct  a decreasing sequence $\{b_k\}$, with  $0 < b_k \leq 1$, such that
\begin{equation}\label{induction}\|g^*\chi_{(b_{k+1},b_k)}\|_{\oX(0,\infty)}> 1
 \end{equation}
for every $k\in \N$. To this purpose, set $b_1=1$, and assume that
$b_k$ is given for some $k\in \mathbb N$. Then  define
$$
h_l =g^*\chi_{\big(\frac{b_k}{l},b_k\big)} \qquad \hbox{for $l \in \N$, with
$l \geq 2$}.
$$
Since $0\leq h_l \nearrow g^*\chi_{(0,b_k)}$,   property  (N3)
tells us that $\|h_{l}\|_{\oX(0,\infty)} \nearrow
\|g^*\chi_{(0,b_k)}\|_{\oX(0,\infty)}$. Inasmuch as
$\|g^*\chi_{(0,b_k)}\|_{\oX(0,\infty)} \geq 2$,  then there exists an $l_0$,
with $l_0 \geq 2$, such that $\|h_{l_0}\|_{\oX(0,\infty)} > 1$. Defining
$b_{k+1}=\frac{b_k}{l_0}$ entails that $0<b_{k+1}<b_k$ and
$\|g^*\chi_{(b_{k+1},b_k)}\|_{\oX(0,\infty)}=\|h_{l_0}\|_{\oX(0,\infty)}> 1$,
as desired.

\noindent Observe that  choosing
$f=g^*\chi_{(0,1)}$, and $a_k=1$,
$I_k=(b_{k+1},b_k)$ for each $k\in \mathbb N$ provides a contradiction to assumption
(H).  Indeed,   inequality \eqref{E:ak},  which agrees with
\eqref{induction} in this case, holds for every $k$, whereas
$\sum_{k=1}^\infty a_k=\infty$. Consequently,~\eqref{E:sufficient} does hold.

Now, take any
$u\in X_{\loc}(\Rn)$ and any non-increasing sequence $\{K_j\}$
of measurable  bounded sets in $\Rn$ such that $\cap_{j\in \N}
K_j=\emptyset$. Clearly,   $u\chi_{K_1} \in X(\Rn)$ and $\displaystyle \lim_{j \to
\infty}\mathcal L^n(K_j)=0$. We may
 assume that $\mathcal L^n (K_1) < 1$, whence $(u\chi_{K_j})^*
\in \overline{X}_1(0,\infty)$ for each $j \in \N$.

From ~\eqref{E:sufficient}  it follows that
\begin{align*}\displaystyle
\lim_{j\to \infty} \|u\chi_{K_j}\|_{X(\Rn)}
&=\lim_{j\to \infty} \|(u\chi_{K_j})^*\|_{\overline{X}(0,\infty)} \leq \lim_{j\to \infty} \|(u\chi_{K_1})^*\chi_{(0,\mathcal
L^n(K_j))}\|_{\overline{X}(0,\infty)}
=0,
\end{align*}
namely
the local absolute continuity of
$\|\cdot\|_{{X(\Rn)}}$.
\end{proof}

\begin{proof}[Proof of Proposition~\ref{P:necessary}] Owing to
 \cite[Corollary 5.6, Chap. 1]{BS}, assertions (i) and (ii) are equivalent. Assertion
(i) follows from Lemmas~\ref{L:absolute-continuity} and~\ref{L:absolute-continuity-bis}.
\end{proof}

\section{The functional $\mathcal G _X$ and the operator $\mathcal M _X$}\label{GM}

This section is devoted to a closer analysis of the functional $\mathcal G _X$ and the operator $\mathcal M _X$ associated with a rearrangement-invariant norm $\|\cdot\|_{{X(\Rn)}}$.
\par We begin with alternate characterizations of the   almost concavity of the
functional $\mathcal G _X$. In what follows, we shall make use of the   fact that
\begin{equation} \label{easy G_X}
 \|h\|_{\overline X(0,\infty)} =  \|h^*\|_{{\overline X(0,\infty)}}= \|(h_*)_*\|_{{\overline X(0,\infty)}} = \mathcal G_{X} (h_*)
\end{equation}
  for every $h \in L^0(0,\infty)$.
\\
Moreover,  by a
\emph{partition} of the interval $(0,1)$ we shall mean a finite collection $\{I_k: k=1, \dots, m\}$, where $I_k=(\tau_{k-1},\tau_k)$ with  $0=\tau_0<\tau_1 <\dots<\tau_m=1$.

\begin{proposition}\label{P:2}
Let $\|\cdot\|_{{X(\Rn)}}$ be a
rearrangement-invariant norm. Then the following conditions are
equivalent:
\begin{enumerate}
\item[\textup{(i)}] the functional $\mathcal G_X$ is  almost concave;
\item[\textup{(ii)}]
  a positive constant $C$ exists  such that
\begin{equation}\label{E:concavity}
\sum_{k=1}^m \mathcal L^1 (I_k) \,\|f \nt_{{\overline X(I_k)}} \leq C\|f\|_{{\overline X(0,\infty)}}
\end{equation} for every $f\in \overline X_1(0,\infty)$, and for every partition $\{I_k: k=1, \dots, m\}$ of $(0,1)$;
\item[\textup{(iii)}]   a positive constant $C$ exists
such that
\begin{equation}\label{E:concavity2}
\sum_{k=1}^m \, \mathcal L^n (B_k) \, \|u\nt_{X(B_k)} \leq C   \ \mathcal L^n \big(\cup_{k=1}^m B_k\big) \ \|u\nt_{X(\cup_{k=1}^m B_k)}
\end{equation}
for every $u\in X_{\loc}(\Rn)$, and  for every finite collection $\{B_k: k=1, \dots, m\}$
of pairwise disjoint balls in $\Rn$.
\end{enumerate}
\end{proposition}

\begin{proof}
\textup{(i)} $\Rightarrow$ \textup{(ii)} Assume that $\mathcal G_X$
is   almost  concave. Fix any function $f\in \oX_1(0,\infty)$,
and any partition $\{I_k: k=1, \dots, m\}$  of $(0,1)$. It is easily
verified that
$$\Big((f\chi_{I_k})^*(\mathcal
L^1(I_k)\cdot)\Big)_* = \frac{(f\chi_{I_k})_*}{\mathcal L^1(I_k)}
\quad \hbox{and} \quad \big(f\chi_{\cup_{k=1}^m I_k}\big)_* =
\sum_{k=1}^m (f\chi_{I_k})_*\,.$$
Hence, by~\eqref{easy G_X} and  the   almost  concavity of $\mathcal
G_{{X}}$, there  exists     a constant $C$ such that
\begin{align*}
\sum_{k=1}^m  \mathcal L^1 (I_k)\, \|f\nt_{{\oX(I_k)}}
&=\sum_{k=1}^m \mathcal L^1 (I_k) \,\|(f\chi_{I_k})^*( \mathcal L^1 (I_k)\cdot)\|_{{\oX(0,\infty)}}
=\sum_{k=1}^m \mathcal L^1 (I_k)\, \mathcal G_{{X}}\Big( \frac{(f\chi_{I_k})_*}{ \mathcal L^1 (I_k)}\Big)\\
&\leq C\, \mathcal G_{X}\Big(\sum_{k=1}^m (f\chi_{I_k})_*\Big)
=C\,\mathcal G_{X}\Big( (f\chi_{\cup_{k=1}^m I_k})_*\Big) \leq
C\,\mathcal G_{X} (f_*) =C\, \|f\|_{\oX(0,\infty)}.
\end{align*}
This yields inequality   \eqref{E:concavity}.
\\
\textup{(ii)} $\Rightarrow$ \textup{(i)} Take any finite collections $\{g_k: k=1, \dots, m\}$  in $\mathcal C$        and
$\{\lambda_k: k=1, \dots, m\}$ in $(0,1)$, respectively, with $\sum_{k=1}\sp m \lambda_k=1$. For each $k=1,\dots,
m$, write $f_k=(g_k)_*$, $a_k= \sum_{i=1}\sp k \lambda_i$, and  $I_k=(a_{k-1},a_k)$ with
$a_0=0$.  Then define
\begin{equation}\nonumber
f(t) =
\sum_{k=1}\sp m
f_k(\tfrac{t-a_{k-1}}{\lambda_k}) \chi_{I_k}(t) \qquad \text{for}\
 t\in (0, \infty).
\end{equation}
Observe that $f_*=\sum_{k=1}\sp m \lambda_k(f_k)_*=\sum_{k=1}\sp m
\lambda_k g_k$. Owing to \eqref{easy G_X} and \eqref{E:concavity},
one thus obtains
\begin{align*}
\mathcal G_X\Big(\sum_{k=1}\sp m\lambda_kg_k\Big) &= \mathcal G_X
(f_*) =\|f\|_{\overline X(0,\infty)} \geq \frac1{C}\sum_{k=1}\sp m
\lambda_k\|f\nt_{\overline X(I_k)}=
\frac1{C}\sum_{k=1}\sp m\lambda_k\|f_k\sp*(\tfrac 1{\lambda_k} \, \cdot)\nt_{\overline X(0,\lambda_k)}\\
&= \frac1{C}\sum_{k=1}\sp m\lambda_k\|f_k\sp*\|_{\oX(0,\infty)}=
\frac1{C}\sum_{k=1}\sp m\lambda_k \G_X\big((f_k)_*\big)=
\frac1{C}\sum_{k=1}\sp m\lambda_k \G_X(g_k),
\end{align*}
whence the  almost concavity of  $\mathcal G_X$ follows.
\\
\textup{(ii)} $\Rightarrow$ \textup{(iii)} Fix any function $u\in
X_{\loc}(\Rn)$, and any finite collection $\{B_k: k=1, \dots, m\}$
of pairwise disjoint balls in $\Rn$. Set  $a_k= \mathcal L^n (B_k)$, for
$k=1,\dots,m$, and $a_0=0$. Define
$$
I_k=\left (\frac{\sum_{j=0}^{k-1} a_j}{\sum_{j=1}^{m} a_j},\frac{\sum_{j=0}^{k} a_j}{\sum_{j=1}^{m} a_j}\right) \qquad \hbox{for $k=1,\dots,m.$}
$$
Thanks to rearrangement-invariance of $\oX(0,\infty)$,  assumption
\textup{(ii)} ensures  that
\begin{align*}
\|u\nt_{X(\cup_{k=1}^m B_k)}
&=\Big\|\big(u\chi_{\cup_{k=1}^m B_k}\big)^*\big(\sum_{k=1}^m a_k \ \cdot\big)\Big\|_{\oX(0,\infty)}
=\Big\|\sum_{k=1}^m (u\chi_{B_k})^*\big(\sum_{j=1}^m a_j \, \cdot\,  - \sum_{j=0}^{k-1} a_j\big) \ \chi_{I_k}\Big\|_{\oX(0,\infty)}\\
&\geq \, \frac{1}{C}\  \sum_{k=1}^m \frac{a_k}{\sum_{j=1}^{m} a_j} \Big\|(u\chi_{B_k})^*\big(\sum_{k=1}^m a_k  \, \cdot\,  - \sum_{j=0}^{k-1} a_j\big) \ \chi_{I_k}  \Big\|^{^{\oslash}}_{\overline X(I_k)}\\
&= \, { \frac{1}{C \, {\sum_{j=1}^{m} a_j} }\,
\sum_{k=1}^m  {a_k}} \, \|(u\chi_{B_k})^*(a_k \,
\cdot)\|_{\overline X(0,\infty)} = \frac{1}{C \, \mathcal L^n
(\cup_{k=1}^m B_k)} \sum_{k=1}^m \, \mathcal L^n (B_k)
\|u\nt_{X(B_k)}.
\end{align*}
Inequality \eqref{E:concavity2} is thus established.
\\
\textup{(iii)} $\Rightarrow$ \textup{(ii)} Assume  that $f\in
\oX_1(0,\infty)$, and that $\{I_k: k=1, \dots, m\}$  is a partition of
$(0,1)$. Let $\{B_k: k=1, \dots, m\}$  be a family of pairwise disjoint
balls in $\Rn$ such that $\mathcal L^n(B_k)=\mathcal L^1(I_k)$, and
let $u$ be a measurable function on $\Rn$ vanishing outside of
$\cup_{k=1}^m B_k$ and fulfilling $(u\chi_{B_k})^*=(f\chi_{I_k})^*$,
for $k=1,2,\dots,m$. Assumption \textup{(iii)} then tells us that
$$
\sum_{k=1}^m \mathcal L^1(I_k) \|f\nt_{\oX(I_k)} \leq C \, \mathcal
L^1(\cup_{k=1}^m I_k) \|f\nt_{\oX(\cup_{k=1}^m I_k)}
=C \, \|f\nt_{\oX(0,1)} = C\,  \|f\|_{\oX(0,\infty)},
$$
namely \eqref{E:concavity}.
\end{proof}

We next focus on the maximal operator $\mathcal M_X$.
Criteria for the validity of the Riesz-Wiener type
inequality~\eqref{E:riesz} are the content of the following result.

\begin{proposition}\label{P:herz}
Let  $\|\cdot\|_{{X(\Rn)}}$ be    a
rearrangement-invariant  norm. Then the following conditions are
equivalent:
\begin{enumerate}
\item[\textup{(i)}]  the Riesz-Wiener type inequality~\eqref{E:riesz} holds for some positive constant $C$, and
for every $u\in X_{\loc}(\Rn)$;
\item[\textup{(ii)}]  a positive constant $C_1$ exists such that
\begin{equation}\label{E:herz}
\min_{k=1,\dots,m} \|u\nt_{{X(B_k)}} \leq C_1\|u\nt_{{X(\cup_{k=1}^m
B_k)}}
\end{equation}   for
  every   $u\in X_{\loc}(\Rn)$, and for  every finite collection $\{B_k: k=1, \dots, m\}$
of pairwise disjoint balls in $\Rn$;
\item[\textup{(iii)}] a positive constant $C_2$ exists such that
\begin{equation*}
\min_{k=1,\dots,m} \|f\nt_{{\oX(I_k)}} \leq
C_2\|f\|_{{\oX(0,\infty)}}
\end{equation*}
for every $f\in \oX_1(0,\infty)$, and for every partition
 $\{I_k: k=1, \dots, m\}$ of $(0,1)$.
\end{enumerate}
\end{proposition}

\begin{proof}
\textup{(i)} $\Rightarrow$ \textup{(ii)} Let
 $u\in X_{\loc}(\Rn)$,  and let $\{B_k: k=1, \dots, m\}$ be a    collection
of pairwise disjoint balls in $\Rn$. When $\min_{k=1,\dots,m} \|u\nt_{X(B_k)}=0$, then~\eqref{E:herz} trivially holds. Assume that $\min_{k=1,\dots,m} \|u\nt_{X(B_k)} >0$. Fix any $s \in \big(0,
\mathcal L^n (\cup_{k=1}^m B_k  )\big)$, and  any $t \in
\big(0,\min_{k=1,\dots,m} \|u\nt_{X(B_k)}\big)$. If $x\in B_j$ for
some $j=1,\dots,m$, then
$$
\M_{X}(u\chi_{\cup_{k=1}^m B_k})(x) \geq  \|u\chi_{\cup_{k=1}^m
B_k}\nt_{X(B_j)} \geq \min_{k=1,\dots,m} \|u\nt_{X(B_k)}>t.
$$
Thus,
$$
\mathcal L^n(\{x\in \Rn: \M_{X}(u\chi_{\cup_{k=1}^m B_k})(x) >t\})
\geq  \mathcal L^n(\cup_{k=1}^m B_k)>s,
$$
and, consequently,
$$
\big(\M_{X}(u\chi_{\cup_{k=1}^m B_k})\big)^*(s) \geq  t.
$$
Since the last inequality holds for every $t<\min_{k=1,\dots,m}
\|u\nt_{X(B_k)}$, one infers that
\begin{equation}\label{E:rearrangement}
\big(\M_{X }(u\chi_{\cup_{k=1}^m B_k})\big)^*(s)\ \geq\
\min_{k=1,\dots,m} \|u\nt_{X(B_k)}.
\end{equation}
On the other hand, an application of assumption \textup{(i)} with
$u$ replaced by $u\chi_{\cup_{k=1}^m B_k}$ tells us that
\begin{equation}\label{E:assumption}
\big(\M_{X}(u\chi_{\cup_{k=1}^m B_k})\big)^*(s) \leq C \,
\|(u\chi_{\cup_{k=1}^m B_k})^*\nt_{\oX(0,s)} \qquad \hbox{for $s \in
\big(0, \mathcal L^n (\cup_{k=1}^m B_k )\big)$.}
\end{equation}
Coupling~\eqref{E:rearrangement}
with~\eqref{E:assumption} implies that
\begin{equation*}
\min_{k=1,\dots,m} \|u\nt_{X(B_k)} \leq C \, \|(u\chi_{\cup_{k=1}^m B_k})^*\nt_{\oX(0,s)} \qquad \hbox{for
$ s \in\big(0, \mathcal L^n (\cup_{k=1}^m B_k  )\big)$. }
\end{equation*}
Thus, owing to the continuity of the function
$s\mapsto \|(u\chi_{\cup_{k=1}^m B_k})^*\nt_{\oX(0,s)}$, which is
guaranteed by Lemma~\ref{T:lemma_conc},
we  deduce that
$$
\min_{k=1,\dots,m} \|u\nt_{X(B_k)} \leq C \|(u\chi_{\cup_{k=1}^m
B_k})^*\nt_{\oX(0,\mathcal L^n(\cup_{k=1}^m B_k))} = C
\|u\nt_{{X(\cup_{k=1}^m B_k)}},
$$
namely \eqref{E:herz}.
\\
 \textup{(ii)} $\Rightarrow$ \textup{(i)}
By \cite[Proposition 3.2]{MP}, condition \textup{(ii)} implies the
existence of a constant $C'$ such that
\begin{equation}\label{E:lerner}
(\M_{X }u)^*(s) \leq C'   \|u^*\big( 3^{-n} {s} \,
\cdot\big){\chi _{(0,1)}(\,\cdot \,)
\|_{\oX(0,\infty)} }  \qquad \hbox{for $s \in (0, \infty)$, }
\end{equation}
for every $u\in X_{\loc}(\Rn)$. By the boundedness of the dilation
operator on rearrangement-invariant spaces, there exists a constant
$C''$, independent of $u$, such that
 \begin{align}\label{E:dilation}
 \|u^*\big( 3^{-n} {s} & \, \cdot\big) {\chi _{(0,1)}(\,\cdot \,)}
 \|_{\oX(0,\infty)} \leq C'' \|u^*(s \, \cdot)
{\chi _{(0,1)}(3^n\,\cdot \,)} \|_{\oX(0,\infty)}\\ \nonumber & \leq
C'' \|u^*(s \, \cdot) {\chi _{(0,1)}(\,\cdot \,)} \|_{\oX(0,\infty)}
= C'' \|u^*\nt_{\oX(0,s)} \qquad \qquad \hbox{for $s \in (0,
\infty)$.}\end{align}
 Inequality~\eqref{E:riesz} follows  from
\eqref{E:lerner} and \eqref{E:dilation}.
\\
\textup{(ii)} $\Leftrightarrow$ \textup{(iii)} The proof is
completely analogous  to that of the equivalence  between conditions \textup{(ii)}
and \textup{(iii)} in  Proposition~\ref{P:2}. We omit
the details for brevity.
\end{proof}

Condition (H) introduced  in Lemma~\ref{L:absolute-continuity} can
be characterized in terms of the maximal operator $\M_{X}$ as
follows.

\begin{proposition}\label{P:3}
Let $\|\cdot\|_{{X(\Rn)}}$ be  a rearrangement-invariant   norm.
Then the  following  assertions are equivalent:
\begin{enumerate}
\item[\textup{(i)}]  $\|\cdot\|_{{X(\Rn)}}$ fulfils condition \textup{(H)}  in Lemma~\ref{L:absolute-continuity};
\item[\textup{(ii)}]   For every function $u\in X(\Rn)$, supported  in
 a set of  finite  measure,
$$\mathcal L^n(\{x\in\Rn\colon \M_{X}u(x)>1\}) \ < \infty.$$
\end{enumerate}
\end{proposition}
\begin{proof}
\textup{(i)} $\Rightarrow$ \textup{(ii)} Let $u \in
X(\Rn)$ be  supported in a set of finite measure. Set  $E=\{x \in\Rn :
\M_{{X}}u(x) > 1 \}.$   According to   \eqref{HL}, for any $y \in E$,
there exists a ball $B_y$ in $\Rn$ such that  $y \in B_y$  and   $
\|u\nt_{{X(B_y)}} >1$. Define   \begin{equation}\label{levelHL}E_1=\Big\{y \in E: \mathcal L^n(B_y)
> \max\{1,\mathcal L^n(\{|u|>0\})\}\Big\}.\end{equation}
We claim that, if $y \in E_1$, then
\begin{equation}\label{april1}
(u\chi_{B_y})^*\big(\mathcal L^n(B_y) s\big) \leq
(u\chi_{B_y})^*(s)\chi_{\big(0,\frac{\mathcal
L^n(\{|u|>0\})}{\mathcal L^n(B_y)}\big)}(s)\qquad \hbox{for $s \in
(0, \infty)$.}
\end{equation}
Indeed, since $\mathcal L^n(B_y) \geq 1$, by the monotonicity of the
decreasing rearrangement
$$(u\chi_{B_y})^*(s) \geq (u\chi_{B_y})^*\big(\mathcal L^n(B_y)
s\big) \qquad \hbox{for $s \in (0, \infty)$.}$$ Thus, inequality
\eqref{april1} certainly holds if $s \in \big(0, \tfrac{\mathcal
L^n(\{|u|>0\})}{\mathcal L^n(B_y)} \big]$. On the other hand,
$\mathcal L^n(B_y) \geq \mathcal L^n(\{|u|>0\})$, and since
$(u\chi_{B_y})^*(s) =0$ for $s \geq \mathcal L^n(\{|u|>0\})$, we
have that  $(u\chi_{B_y})^*\big(\mathcal L^n(B_y) s\big)=0$ if $s
\in \big(\tfrac{\mathcal L^n(\{|u|>0\})}{\mathcal L^n(B_y)},
\infty\big)$. Thereby, inequality \eqref{april1} also holds for
these values of $s$.
%
%
%
%
%
%
%
%
%
 Owing to \eqref{april1},
\begin{align}\label{E:estimate-measure}
1 & <\|u\nt_{X(B_y)} =\|(u\chi_{B_y})^*(\mathcal L^n(B_y) \
\cdot)\|_{\oX(0,\infty)} \\
\nonumber &  \leq \Big\|(u\chi_{B_y})^*\chi_{\big(0,\frac{\mathcal
L^n(\{|u|>0\})}{\mathcal L^n(B_y)}\big)}\Big\|_{\oX(0,\infty)}\leq
\Big\|u^*\chi_{\big(0,\frac{\mathcal L^n(\{|u|>0\})}{\mathcal
L^n(B_y)}\big)}\Big\|_{\oX(0,\infty)}.
\end{align}
 Since  (i) is in force, equation \eqref{E:sufficient}    holds with
$g=u^*\chi_{(0,1)} \in \oX_1(0,\infty)$, namely,
$$
\lim_{t\to 0^+} \|u^*\chi_{(0,t)}\|_{\oX(0,\infty)}=0.
$$
 This implies the existence of some  $t_0\in (0,1)$ such that
$\|u^*\chi_{(0,t)}\|_{\oX(0,\infty)} <1$
 for every $t\in (0,t_0)$. Thus,  \eqref{E:estimate-measure} entails that
$$
 \frac{\mathcal L^n(\{|u|>0\})}{\mathcal L^n(B_y)}  \geq t_0
$$
for every $y \in E_1$. Hence,   by \eqref{levelHL},
\begin{equation}\label{palle}\sup_{y\in E} \, {\mathcal L^n(B_y)}\leq \max \Big\{1,\frac{\mathcal L^n(\{|u|>0\})}{t_0}\Big\}.\end{equation}
An application of   Vitali's covering lemma, in the form of
\cite[Lemma 1.6, Chap. 1]{SteinBook}, ensures that there exists  a
countable set $\mathcal I\subseteq E$ such that the family
$\{B_y: y \in \mathcal I\}$ consists of pairwise disjoint balls,
such that $E\subseteq \cup_{y\in \mathcal  I} 5 B_y$. Here, $5 B_y$ denotes
the ball, with the same center as $B_y$, whose radius is $5$ times
the radius of $B_y$.  If $\mathcal I$ is finite, then
trivially $\mathcal L ^n (E) \leq 5^n\sum_{y\in \mathcal I} \mathcal L^n(B_y)
< \infty$. Assume that, instead, $\mathcal I$ is infinite, and  let
$\{y_k\}$ be the sequence of its elements.
For each $k\in \N$,  set, for simplicity, $B_k=B_{y_k}$, and
\begin{equation}
\alpha_k=\mathcal L^n(\{y\in B_k: u(y)\neq 0\}),
\qquad
 I_k=\left(\frac{\sum_{i=0}^{k-1} \alpha_i}{\alpha}, \frac{\sum_{i=1}^{k} \alpha_i}{\alpha}\right),
\qquad a_k=\frac{\mathcal L\sp n(B_k)}{\alpha},
 \end{equation}
 where $\alpha=\mathcal L^n(\{y\in \Rn: u(y)\neq 0\})$ and $\alpha_0=0$.
Note that  $\{I_k\}$ is a sequence of pairwise
disjoint intervals  in $(0,1)$, and
$a_k\geq \mathcal L^1(I_k)$ for each
$k\in \N$. The function $f : (0, \infty ) \to [0, \infty)$, defined by
$$
f(s)=\sum_{k=1}^\infty \big(u\chi_{\{x\in B_k: u(x)\neq
0\}}\big)^*\Big(\alpha s-\sum_{i=1}^{k-1} \alpha_i\Big)  \chi_{I_k}(s)  \quad
\hbox{for $s\in (0,\infty)$,}
$$
   belongs  to $\oX_1(0,\infty)$, and
\begin{align*}
\|(f\chi_{I_k})^*\nt_{\oX(0,a_k)}
&=\|(u\chi_{\{x\in B_k: u(x)\neq 0\}})^*(\alpha \, \cdot)\nt_{\oX(0,a_k)}\\
&=\|(u\chi_{B_k})^*(\mathcal L^n(B_k) \, \cdot)\|_{\oX(0,\infty)}
=\|u\nt_{X(B_k)} >1.
\end{align*}
By {(i)}, one thus obtains that  $\sum_{k=1}^\infty \mathcal
L^n(B_k)=\alpha\sum_{k=1}^\infty a_k <\infty$. Hence $\mathcal L^n(E)
\leq 5^n \sum_{k=1}^\infty \mathcal L^n(B_k) < \infty$,   also in this
case.
\\
 \textup{(ii)} $\Rightarrow$ \textup{(i)} Let $f\in \oX _1(0,\infty)$,  let $\{I_k\}$
 be any sequence of pairwise disjoint intervals in $(0,1)$, and let $\{a_k\}$ be a sequence of
  real numbers, such that $a_k \geq \mathcal L^1(I_k)$, fulfilling \eqref{E:ak}.

\noindent Consider any
  sequence $\{B_k\}$ of pairwise disjoint balls in $\Rn$, such that $\mathcal L^n(B_k)=a_k$ for  $k\in \N$.
  For each $k\in \N$, choose  a function $g_k : \Rn \to [0, \infty)$, supported in $B_k$, and such that $g_k$ is equimeasurable with $f\chi_{I_k}$.
Then, define $u=\sum_{k=1}^\infty g_k$. Note that $u\in
X(\Rn)$, since $u^*=(f\chi_{\bigcup_{k=1}^\infty I_k})^*\leq f^*$.
Furthermore, $u$  is supported in a set of finite measure. Thus,
assumption \textup{(ii)} implies that
\begin{equation}\label{finite measure of M}
\mathcal L^n(\{x \in\Rn : \M_{{X}}u(x) > 1 \}) < \infty.
\end{equation}
If $x\in B_k$ for some $k\in \N$, then
\begin{equation}\label {1000}
\M_{{X}}u(x) \geq  \|u\nt_{X(B_k)} = \|g_k\nt_{X(B_k)}= \|g_k^*\nt_{\oX(0,\mathcal L^n(B_k))}=\| (f\chi_{I_k})^*\nt_{{\oX(0,a_k)}}> 1.
\end{equation}
Consequently,
$$
\cup_{k=1}^\infty B_k \subseteq \{x\in \Rn: \M_{{X}}u(x)>1\}
$$
and
$$
\sum_{k=1}^\infty a_k =\mathcal L^n(\cup_{k=1}^\infty B_k) \leq \mathcal L^n(\{x \in\Rn : \M_{{X}}u(x) > 1 \}) <\infty.
$$
Condition \textup{(i)} is thus fulfilled.
\end{proof}

\section{Proofs of  Theorems~\ref{T:main-theorem1},~\ref{T:main-theorem2} and~\ref{T:main-theorem3}}\label{proofsmain}

%
%

The core of Theorems~\ref{T:main-theorem1},~\ref{T:main-theorem2}
and~\ref{T:main-theorem3} is contained  in the following statement.

\begin{proposition}\label{L:v}
Given    a  rearrangement-invariant norm $\|\cdot\|_{{X(\Rn)}}$,  consider the following
properties:
\begin{enumerate}
\item[\textup{(i)}] $\|\cdot\|_{X(\Rn)}$ satisfies  the Lebesgue point property;

\item[\textup{(ii)}]  $\|\cdot\|_{X(\Rn)}$ fulfills condition \textup{(H)} of Lemma~\ref{L:absolute-continuity};

\item[\textup{(iii)}] The functional $\mathcal G_X$  is almost concave;

\item[\textup{(iv)}] The Riesz-Wiener type inequality~\eqref{E:riesz} holds for some positive constant $C$, and
for every $u\in X_{\loc}(\Rn)$;

\item[\textup{(v)}] The operator $\M_{X}$ is of weak type from $X_{\loc}(\Rn)$ into $L^1_{\loc}(\Rn)$.
\end{enumerate}

\noindent Then:

\medskip
\centerline{\textup{(i)} $\Rightarrow$ \textup{(ii)} $\Rightarrow$
\textup{(iii)} $\Rightarrow$ \textup{(iv)} $\Rightarrow$
\textup{(v)}.}
\medskip
\noindent
 If, in addition, $\|\cdot\|_{X(\Rn)}$ is locally
absolutely continuous, then
\medskip

\centerline{\textup{(v)} $\Rightarrow$ \textup{(i)}.}

\end{proposition}

A proof of Proposition~\ref{L:v} requires  the next lemma.

\begin{lemma}\label{T:lemma}
Let $\|\cdot\|_{{X(\Rn)}}$ be  a  rearrangement-invariant  norm such
that
\begin{equation}\label{limphi}
\lim _{s \to 0^+}\varphi_{{X(\Rn)}}(s) =0.
\end{equation}
If $u : \Rn \to \R$  is any
simple function, then
$$
\lim_{r\to 0^+} \|u-u(x)\nt_{{X(B_r(x))}}=0 \qquad
\hbox{for   a.e. $x \in \Rn$}.
$$
\end{lemma}
\begin{proof}
Let $E$  be  any measurable subset of $\Rn$. By the Lebesgue density
theorem,
\begin{equation}
\begin{aligned}\label{E:in}
\lim_{r\to 0^+} \frac{\mathcal L^n(B_r(x)\setminus E)}
{\mathcal L^n (B_r(x))} &=0 \qquad
{\hbox{for   a.e. $x \in E$}}\\
\lim_{r\to 0^+} \frac{  \mathcal L^n(B_r(x)\cap E)}{\mathcal  L ^n
(B_r(x))}&=0 \qquad {\hbox{for  a.e. $x\in \Rn \setminus
E$}}.
\end{aligned}
\end{equation}
 Since
\begin{equation}\label{char}
 \lim_{r\to 0^+}  \|\chi_{{E}}-\chi_{{E}} (x)\nt_{{X(B_r(x))}}  =
 \begin{cases}\displaystyle
\lim_{r\to 0^+}  \Big\|\chi_{\left(0,  \frac{\mathcal L^n(B_r(x)\setminus E)}
{\mathcal L^n (B_r(x))} \right)} \Big\|_{\overline X(0,\infty)} \quad \hbox{for   a.e. $x \in E$,}\\
\displaystyle \lim_{r\to 0^+} \Big\|\chi_{\left(0, \frac{  \mathcal L^n(B_r(x)\cap E)}{\mathcal  L ^n
(B_r(x))}\right)}\Big\|_{\overline X(0,\infty)} \quad    \hbox{for
 a.e. $x\in \Rn \setminus E$,}
\end{cases}
\end{equation}
 it follows from \eqref{limphi} that
\begin{equation}\label{april2}
\lim_{r\to 0^+}  \|\chi_{{E}}-\chi_{{E}} (x)\nt_{{X(B_r(x))}} =0
\quad \hbox{for  a.e. $x\in \Rn$.}
\end{equation}
Hence, if $u$ is any simple function having the form $u=\sum_{i=1}^k
a_i\chi_{{E_i}}$, where  $E_1, \dots, E_k$  are   pairwise disjoint
measurable  subsets of $\Rn$, and   $a_1, \dots, a_k \in \mathbb R$,
then
\begin{equation}\label{simple}
\lim_{r\to 0^+} \|u-u(x)\nt_{{X(B_r(x))}} \, \leq \, \lim_{r\to 0^+}
\sum_{i=1}^k  \, |a_i|\, \|\chi_{{E_i}} -
\chi_{{E_i}}(x)\nt_{{X(B_r(x))}} = 0
\end{equation}
for   a.e. $x\in \Rn$.
\end{proof}

\begin{proof}[Proof of  Proposition~\ref{L:v}]\,
\textup{(i)} $\Rightarrow$ \textup{(ii)} This  is just the content
of Lemma~\ref{L:absolute-continuity}  above.
\\
\textup{(ii)} $\Rightarrow$ \textup{(iii)}   We prove this
implication by contradiction. Assume that the functional $\mathcal G
_X$ is not almost concave. Owing to
Proposition~\ref{P:2}, this amounts to assuming that, for every $k
\in \N$, there exist a function $f_k \in \overline X_1(0,\infty)$
and a partition $\{J_{k,l}: l=1,\dots,m_k \}$ of $(0,1)$ such that
\begin{equation}\label{101}
\sum_{l=1}^{m_k} \mathcal L^1(J_{k,l}) \, \|f_k\nt_{{\oX(J_{k,l})}}  \ > \ 4^k \|f_k\|_{{\overline X(0,\infty)}}.
\end{equation}
Define the function $f : (0, \infty ) \to \mathbb R$ as
\begin{equation}\label{102a}
f(t) =\sum_{k=1}^\infty \frac{\chi_{(2^{-k}, 2^{-k+1})}(t)  \,
f_k(2^k t-1)}{2^k \, \big\|\chi_{(2^{-k}, 2^{-k+1})}\, f_k(2^k
\cdot-1)\big\|_{{\overline X(0,\infty)}}} \qquad {\hbox{for $t\in (0, \infty ) $}}.
\end{equation}
Since  $f \in L^0(0,\infty)$, $f=0$ on $(1,\infty)$ and
 $
\|f\|_{{\overline X(0,\infty)}}
\leq\sum_{k=1}^\infty   2^{-k}
 $, we have that $f \in \overline X_1(0,\infty)$.
\\
Let us denote by $\Lambda$ the set $\{(k,l)\in \N^2: l\leq m_k\}$,
ordered   according to the  lexicographic order, and define the
sequence $\{I_{k,l}\}$ as
\begin{equation}\label{intervals}
I_{k,l}=\frac{1}{2^k} J_{k,l} + \frac{1}{2^k}  \qquad
{\hbox{for   $(k, l)\in \Lambda$}}.
\end{equation}
Each element $I_{k,l}$ is  an open subinterval of $(0,1)$. Moreover,
the intervals  $I_{k,l}$ and $I_{h,j}$ are disjoint   if $(k,l)\neq
(h,j)$. Actually, if $k\neq h$, then
$$
I_{k,l} \cap I_{h,j} \ \subseteq \ (2^{-k},2^{1-k}) \cap ({2^{-h}}, {2^{1-h}})\  = \ \emptyset ;
$$
if, instead, $k= h$ but $l\neq j$, then the same conclusion immediately
follows  from the fact that the intervals $J_{k,l}$
and $J_{k,j}$ are disjoint.
 Owing to \eqref{101} and \eqref{102a},
\begin{align}\label{202}
{\sum_{(k,l)\in \Lambda}}  \mathcal L^1(I_{k,l}) \, \|f\nt_{{\overline X(I_{k,l})}}
&= {\sum_{(k,l)\in \Lambda}}   \mathcal L^1(I_{k,l}) \, \|(f\chi_{I_{k,l}})^*( \mathcal L^1(I_{k,l}) \, \cdot)\|_{{\overline X(0,\infty)}}\\
&\nonumber= {\sum_{(k,l)\in \Lambda}}  \frac{\mathcal L^1(J_{k,l})}{2^k} \,  \frac{\|(f_k\chi_{J_{k,l}})^*(\mathcal L^1(J_{k,l}) \, \cdot)\|_{{\overline X(0,\infty)}}}{2^k\|\chi_{({2^{-k}}, {2^{1-k}})}f_k(2^k \,\cdot - 1)\|_{{\overline X(0,\infty)}}}\\
&\nonumber=\sum_{k=1}^\infty \frac{1}{4^k\|f_k^*(2^k \, \cdot)\|_{{\overline X(0,\infty)}}} \sum_{l=1}^{m_k} \mathcal L^1(J_{k,l}) \, \|f_k\nt_{{\overline X(J_{k,l})}}\\
&\nonumber\geq  \sum_{k=1}^\infty \frac{4^k \|f_k\|_{{\overline
X(0,\infty)}}}{4^k \|f_k^*(2^k \,\cdot)\|_{{\overline X(0,\infty)}}}
\geq  \sum_{k=1}^\infty \frac{4^k \|f_k\|_{{\overline
X(0,\infty)}}}{4^k \|f_k^*\|_{{\overline X(0,\infty)}}} \, = \,
\sum_{k=1}^\infty 1 =\infty.
\end{align}
Set $M=\{(k,l) \in \Lambda:  \|f\nt_{{\overline X(I_{k,l})}} \leq
2\}$, and observe that
\begin{equation}\label{103}
\sum_{(k,l)\in M} \mathcal L^1(I_{k,l})\,  \|f\nt_{{\overline X(I_{k,l})}}
\leq 2\sum_{(k,l)\in M}  \mathcal L^1(I_{k,l}) \leq 2.
\end{equation}
  From \eqref{202} and \eqref{103} we thus infer that
$$
\sum_{(k,l)\in \Lambda \setminus M}  \mathcal L^1(I_{k,l})  \,  \|f\nt_{{\overline X(I_{k,l})}}
=\infty.
$$
On the other hand,   assumption (ii) implies property
\eqref{E:sufficient}. This property,  applied with
$g=f\chi_{I_{k,l}}$, in turn ensures that,   for every $(k,l)\in
\Lambda$,
$$
\lim_{t\to + \infty} \|(f\chi_{I_{k,l}})^* \nt_{{\overline  X(0,t)}}
=\lim_{t\to + \infty} \|(f\chi_{I_{k,l}} )^*(t \, \cdot)\|_{{\overline X(0,\infty)}}
\leq \lim_{t \to +\infty} \|(f\chi_{I_{k,l}})^* \, \chi_{(0,\frac 1t)}\|_{{\overline X(0,\infty)}}=0.
$$
Note that the inequality holds since
the function $(f\chi_{I_{k,l}})^*$ belongs to $\overline X_1(0,
\infty)$, and is non-increasing, and hence $(f\chi_{I_{k,l}} )^*(t
s)\leq (f\chi_{I_{k,l}})^*(s) \, \chi_{(0,\frac 1t)}(s)$ for $s
\in (0, \infty)$.

\noindent Thus, owing to
 Lemma~\ref{T:lemma_conc}, if
 $(k,l)\in \Lambda \setminus M$,  there
exists a number $a_{k,l}\geq \mathcal L^1(I_{k,l}) $ such that
\begin{equation}\label{104}\|(f\chi_{I_{k,l}})^*\nt_{{\overline
X(0,a_{k,l})}}=2\,.
\end{equation}
Furthermore, by \eqref{104},
\begin{align}\label{104bis}
\sum_{(k,l)\in \Lambda \setminus M} a_{k,l} &=
\frac{1}{2}\sum_{(k,l)\in \Lambda \setminus M} a_{k,l} \,
\|(f\chi_{I_{k,l}})^*\nt_{{\overline X(0,a_{k,l})}} \geq \frac{1}{2}
\sum_{(k,l)\in \Lambda \setminus M} \mathcal L^1(I_{k,l}) \,
\|(f\chi_{I_{k,l}})^*\nt_{{\overline X(0,\mathcal L^1(I_{k,l}))}} \\
\nonumber
 &= \frac{1}{2} \sum_{(k,l)\in \Lambda \setminus M} \mathcal L^1(I_{k,l}) \, \|f\nt_{{\overline X(I_{k,l})}} =\infty.
\end{align}
Thanks  to \eqref{104bis}, the function $f \in \overline
X_1(0,\infty)$, defined by \eqref{102a}, the  sequence
$\{I_{k,l}\}$,   defined by
\eqref{intervals}, and the  sequence $\{a_{k,l}\}$ contradict  condition (H) in
Lemma~\ref{L:absolute-continuity}, and, thus, assumption~\textup{(ii)}.
\\
\textup{(iii)} $\Rightarrow$ \textup{(iv)}  This implication follows
from Propositions~\ref{P:2} and~\ref{P:herz}, since condition
\textup{(ii)} of Proposition~\ref{P:2} trivially implies condition
\textup{(iii)} of Proposition~\ref{P:herz}.
\\
\textup{(iv)} $\Rightarrow$ \textup{(v)} Let $K$ be a bounded subset
of $\Rn$.  Fix any function $u\in X_{\loc}(\Rn)$
whose support is contained in $K$. Clearly, $u=u\chi_{K}$.
From \textup{(iv)}, we
infer that
\begin{align*}
\sup_{t>0} & t \, \mathcal L^n(\{x \in K : \M_{X} u(x)>t\})
=\sup_{t>0} t \, \mathcal L^n(\{x \in \Rn: \chi_{K}\M_{X} u(x)>t\})
 \\ &  = \sup_{s>0} s\big(\chi_{K}\, \M_{X}u\big)^*(s)
 \leq \sup_{s\in (0, \mathcal L^n(K))} s\big(\M_{X} u \big)^*(s) \leq C \sup_{s\in (0,\mathcal L^n(K))} s\| u^*\nt_{\oX(0,s)}
\\ & \leq C \, \mathcal L^n(K) \, \|u^*\nt_{\oX(0,\mathcal L^n(K))}\leq C' \|u^*\|_{\oX(0,\infty)} =C'\|u\|_{X(\Rn)},
\end{align*}
for some constants $C$ and $C'$,  where the last but one inequality
follows from Lemma~\ref{T:lemma_conc}, and the last one from the
boundedness of the dilation operator on rearrangement-invariant
spaces. Property \textup{(v)} is thus established.

Finally, assume that  $\|\cdot\|_{X(\Rn)}$ is
locally absolutely continuous and satisfies condition \textup{(v)}.
Since $\Rn$ is the  countable union of balls, in order to prove
\textup{(i)} it suffices to show that, given any $u \in
X_{\loc}(\Rn)$ and any ball $B$ in $\Rn$,
\begin{equation}\label{leb in X}
\lim _{r \to 0^+}\    \| u - u(x) \nt_{X(B_r(x))} = 0\qquad
\hbox{for   a.e. $x \in B$}.
\end{equation}
Equation  \eqref{leb in X} will in turn follow if we show that, for
every $t>0$, the set
\begin{equation}\label{E:at}
A_t=\{x\in B: \limsup_{r\to 0^+}\|u-u(x)\nt_{{X(B_r(x))}} >  2t\}
\end{equation}
has  measure zero.   To prove this, we begin by observing that,
since $\|\cdot\|_{{X(\Rn)}}$ is locally absolutely continuous,
\cite[Theorem 3.11, Chap. 1]{BS} ensures that for any
$\varepsilon>0$ there exists a simple function $v_{\varepsilon}$
supported on $B$ such that $u\chi_B=v_{\varepsilon}+w_{\varepsilon}$
and $\|w_{\varepsilon}\|_{{X(B)}}<\varepsilon$. Clearly,
$w_{\varepsilon}$ is supported on $B$ as well. Moreover,
\cite[Theorem 5.5, Part (b), Chap. 2]{BS} and Lemma~\ref{T:lemma}
imply that
$$
\lim_{r\to 0^+} \|v_{\varepsilon}-v_{\varepsilon}(x)\nt_{X(B_r(x))} =0 \qquad
\hbox{for   a.e. $x \in B$}.
$$
Fix any $\varepsilon>0$. Then
%
\begin{align*}
\limsup_{r\to 0^+} \|u-u(x)\nt_{X(B_r(x))}
&\leq \limsup_{r\to 0^+} \|v_{\varepsilon}-v_{\varepsilon}(x)\nt_{X(B_r(x))} + \limsup_{r\to 0^+} \|w_{\varepsilon}-w_{\varepsilon}(x)\nt_{X(B_r(x))}\\
&=\limsup_{r\to 0^+}
\|w_{\varepsilon}-w_{\varepsilon}(x)\nt_{X(B_r(x))} \leq
\M_{{X}}w_{\varepsilon}(x) + |w_{\varepsilon}(x)|\,
\|\chi_{(0,1)}\|_{{\overline X(0,\infty)}}.
\end{align*}
Therefore,
\begin{equation}\label{149}
A_t \, \subseteq \,
 \{x\in B: \M_{{X}}w_{\varepsilon}(x)>   t  \} \, \cup \,  \{y\in B: |w_{\varepsilon}(y)|\, \|\chi_{(0,1)}\|_{{\overline X(0,\infty)}}>  t \}
 \qquad \hbox{for  $t\in (0, \infty)$.}
  \end{equation}
Owing to \textup{(v)},
$$
\mathcal L^n(\{x\in B: \mathcal M_{{X}}w_{\varepsilon}(x)>  t \})
\leq \frac{C}{t} \, \|w_{\varepsilon}\|_{{X(B)}}  \qquad \hbox{for
$t\in (0, \infty)$.}
$$
On the other hand,
\begin{align*}
\mathcal L^n(\{y\in B:
|w_{\varepsilon}(y)|\|\chi_{(0,1)}\|_{{\oX(0,\infty)}}>   t \}) \leq
\frac1t \,  \|\chi_{(0,1)}\|_{\oX(0,\infty)} \,
\|w_{\varepsilon}\|_{L^1(B)}\leq \frac{ C_0 }{t}
\|\chi_{(0,1)}\|_{{\oX(0,\infty)}} \|w_{\varepsilon}\|_{{X(B)}}
\end{align*}
for every $t\in (0, \infty)$, where $C_0$ is the norm of the
embedding $X(B) \to L^1(B)$. Inasmuch as
$\|w_{\varepsilon}\|_{{X(B)}}<\varepsilon$, the last two
inequalities, combined with \eqref{149} and with the subadditivity
of the outer Lebesgue  measure, imply that the outer Lebesgue
measure of $A_t$ does not exceed
$$\frac {\varepsilon}t  \big(C+C_0\|\chi_{(0,1)}\|_{{\oX(0,\infty)}}\big)$$ for every $t\in
(0, \infty)$. Hence, $\mathcal L^n (A_t)=0$, thanks to the
arbitrariness of $\varepsilon >0$.
\end{proof}

\begin{proof}[Proof of Theorem~\ref{T:main-theorem1}]
This is a consequence of Propositions ~\ref{P:necessary} and
~\ref{L:v}.
\end{proof}

\begin{proof}[Proof of Theorem~\ref{T:main-theorem2}]
This is a consequence of Propositions ~\ref{P:necessary} and
~\ref{L:v}.
\end{proof}

\begin{proof}[Proof of Theorem~\ref{T:main-theorem3}] The equivalence of conditions
 \textup{(i)} and \textup{(ii)} follows from   Proposition~\ref{P:3}, Proposition~\ref{L:v} and
Lemma~\ref{L:absolute-continuity-bis}.
\\ In order to verify  the equivalence  of \textup{(ii)} and \textup{(iii)},
it   suffices to observe that, thanks to the positive homogeneity of
the maximal operator $\M_{{X}}$, one has that      $\mathcal L^n(\{x
\in\Rn : \M_{{X}}u(x) > 1 \}) < \infty$ for every $u\in X(\Rn)$
supported in a set of finite measure if, and only if, $\mathcal
L^n(\{x \in\mathbb R^n : \M_{{X}}u(x) > t \}) < \infty$ for every
$u\in X(\Rn)$ supported in a set of finite measure and for  every
$t\in (0, \infty)$.  The latter condition is equivalent to
\textup{(iii)}.
\end{proof}

\section{Proofs of Propositions~\ref{T:Lorentz}--\ref{T:Marcinkiewicz}}\label{S:examples}

In this last section, we show  how our general criteria can be
specialized to characterize those rearrangement-invariant norms,
from customary families,  which satisfy the Lebesgue point property,
as stated in Propositions~\ref{T:Lorentz}--\ref{T:Marcinkiewicz}. In fact, these propositions admit diverse
proofs, based on the different   criteria provided by Theorems~
\ref{T:main-theorem1}, ~\ref{T:main-theorem2} and~\ref{T:main-theorem3}.
For instance, Propositions \ref{T:Lorentz}--\ref{T:Lambda}
  can be derived via
 Theorem~\ref{T:main-theorem2}, combined with results on the local
absolute continuity of the norms in question and  on Riesz-Wiener
type inequalities contained in \cite{BP} (Orlicz norms),
 \cite{BMR} (norms in the Lorentz spaces
$L^{p,q}(\Rn)$), and \cite{MP} (norms in the Lorentz  endpoint
spaces $\Lambda_\varphi(\Rn)$). Let us also mention that, at least
in the one-dimensional case, results from these propositions overlap
with those of ~\cite{BS1, BS2, SS}.
%
%
\par
Hereafter, we provide alternative, more self-contained proofs of
Propositions~\ref{T:Lorentz}--\ref{T:Marcinkiewicz}, relying upon
our general criteria. Let us begin with Proposition~\ref{T:Lorentz},
whose proof requires the following preliminarily lemmas.
%
%

\begin{lemma}\label{L:concavity-lorentz}
Let $p, q \in [1, \infty]$ be admissible values in the definition of
the Lorentz norm \, \mbox{$\|\cdot \|_{L^{p,q}(\Rn)}$.}~Then
\begin{equation}\label{lorentzG}
\G_{L^{p,q}}(f)=
\begin{cases}
\left(p\int_0^\infty s^{q-1} (f(s))^{\frac{q}{p}}\,d \mathcal L^1 (s)\right)^{\frac{1}{q}} & \mbox{if} \  \ 1<p<\infty \mbox{ and}\  1\leq q<\infty, \mbox{or}\  p=q=1;\\
\displaystyle \sup_{s\in (0,\infty)} s (f(s))^{\frac{1}{p}} & \mbox{if}\  \ 1< p<\infty \ \mbox{and} \ q=\infty;\\
\mathcal L^1(\{s\in (0,\infty): f(s)>0\}) &\mbox{if}\  \  p=q=\infty\,,
\end{cases}
\end{equation}
for every
 non-increasing function $f: [0,\infty)
\rightarrow [0,\infty]$. Hence, the functional $\G_{L^{p,q}}$ is
  concave if $1 \leq q \leq p$.
\end{lemma}

\begin{proof} Equation \eqref{lorentzG}
follows from a well-known expression of Lorentz norms in terms of
the distribution function (see, e.g.,~\cite[Proposition 1.4.9]{G}),
from  equality~\eqref{easy G_X} and from the fact that every
non-increasing function $f:[0,\infty)\rightarrow [0,\infty]$
agrees a.e. with the function
$f=(f_*)_*$.
\\ The fact
that $\G_{L^{p,q}}$ is concave if $1 \leq q \leq
p$ is an easy consequence of  the representation formulas
\eqref{lorentzG}. In particular, the fact that the function $[0,
\infty) \ni t \mapsto t^\alpha$ is concave if $0 < \alpha \leq 1$ plays a role here.
\end{proof}
%

\begin{lemma}\label{L:non-concavity-Lorentz}
Suppose that $1\leq p<q<\infty$. Then there exists a function $u\in
L^{p,q}(\Rn)$, having support of finite measure, such that
\begin{equation}\label{E:orlicz}
\mathcal L^n(\{x\in \Rn: \mathcal M_{L^{p,q}} u(x) >1\})=\infty.
\end{equation}
\end{lemma}

\begin{proof}
We shall prove that the norm
$\|\cdot \|_{L^{p,q}(\Rn)}$ does not satisfy condition \textup{(H)}
from Lemma~\ref{L:absolute-continuity},  if $1\leq p<q<\infty$. The
conclusion will then follow via Proposition~\ref{P:3}.
\\
To this purpose, define $f : (0,
\infty) \to [0, \infty)$ as
\begin{equation}\label{E:def}
f(s)= c \sum_{k=1}^\infty \left(\frac{3^k}{k}\right)^{\frac1p}
\chi_{_{\left(\frac{1}{2\cdot 3^k},\frac{1}{2\cdot
3^{k-1}}\right)}}(s) \qquad \text{for} \ s \in (0,\infty),
\end{equation}
where $c>\big(\frac qp\big)^{\frac 1q}$. Observe that $f=
f^*\chi_{(0,1)}$ a.e., since $f$ is  a nonnegative decreasing
function in $(0,\infty)$ with support in $(0,1)$. Moreover,
$f \in {L^{p,q}(0,\infty)}$, since
\begin{equation}\nonumber
\left\|f\right\|_{L^{p,q}(0,\infty)}^q
= c^q \sum_{k=1}^\infty \int_{\frac{1}{2\cdot 3^k}}^{\frac{1}{2\cdot 3^{k-1}}} \left(\frac{3^k}{k}\right)^{\frac qp} s^{\frac qp-1}\,d \mathcal L^1 (s)\\
\leq c^q \sum_{k=1}^\infty  \left(\frac{3^k}{k}\right)^{\frac qp}
\int_{0}^{\frac{1}{2\cdot 3^{k-1}}} s^{\frac qp-1}\,d \mathcal L^1
(s)
 <   \infty\,,
\end{equation}
thanks to the assumption that $q>p$. For each $k \in \N$, set
$I_k=\left(\frac{1}{2 \cdot 3^k},\frac{1}{2 \cdot 3^{k-1}}\right)$
and $a_k=\frac1k$. Then
\begin{align*}
\|(f\chi_{I_k})^*\nt_{_{L^{p,q}(0,a_k)}} = c \left(\int_0^\infty
\left(\frac{3^k}{k}\right)^{\frac qp} \chi_{(0,\frac{1}{3^k})}
\left(\frac sk\right) s^{\frac qp-1}\,d \mathcal L^1 (s)
\right)^{\frac 1q} =c \left(\frac pq\right)^{\frac 1q} >1,
\end{align*}
and hence   condition (H) of Lemma~\ref{L:absolute-continuity} fails
for the  norm $\|\cdot\|_{L^{p,q}(\Rn)}$.
\end{proof}

We are now in a position to accomplish the proof of
Proposition~\ref{T:Lorentz}.

\par\noindent
\begin{proof}[Proof of Proposition~\ref{T:Lorentz}.] By Lemma~\ref{L:concavity-lorentz}, the
 functional $\G_{L^{p,q}}$ is concave if $1 \leq q \leq p$. Moreover, the norm $\|\cdot \|_{L^{p,q}(\Rn)}$ is locally absolutely continuous
 if and only if
  $q< \infty$  -- see e.g.~ \cite[Theorem~8.5.1]{PK}.  Thereby, an application of Theorem \ref{T:main-theorem1}
   tells us that,  if $1 \leq q \leq p<\infty$, then the norm $\|\cdot \|_{L^{p,q}(\Rn)}$ has the Lebesgue point property.
\\ On the other hand, coupling
 Theorem~\ref{T:main-theorem3} with
Lemma~\ref{L:non-concavity-Lorentz} implies that the norm $\|\cdot
\|_{L^{p,q}(\Rn)}$ does not have the Lebesgue point property if $1 \leq
p<q<\infty$. \\ In the remaining case when $q=\infty$, the norm
$\|\cdot \|_{L^{p,q}(\Rn)}$ is not locally absolutely continuous.
Hence, by Theorem~\ref{T:main-theorem1}, it does not have the
Lebesgue point property.
\end{proof}

One proof of Proposition~\ref{T:Orlicz}, dealing with Orlicz norms,
will follow from Theorem~\ref{T:main-theorem3}, via the next lemma.

\begin{lemma}\label{L:orlicz}
Let $A$ be a Young function satisfying the $\Delta_2$-condition near
infinity. Then
\begin{equation*}
\mathcal L^n(\{x\in \Rn: \mathcal M_{L^A} u(x) >1\})<\infty
\end{equation*}
for every $u\in L^A(\Rn)$,   supported in a set of finite measure.
\end{lemma}

\begin{proof}
Owing to Proposition~\ref{P:3}, it suffices to show that condition
\textup{(H)} from Lemma~\ref{L:absolute-continuity} is fulfilled by
the Luxemburg norm.
\\
Consider any function $f\in L^A_1(0,\infty)$,    any  sequence
$\{I_k\}$ of pairwise disjoint intervals in $(0,1)$, and any
sequence $\{a_k\}$ of positive real numbers such that
$$
a_k \geq \mathcal L^1(I_k) \quad \text{and} \quad
\|(f\chi_{I_k})^*\nt_{L^A(0,a_k)} > 1 \quad \hbox{for $k\in \N$.}
$$
Since
\begin{equation}\label{orlicz1}\nonumber
{a_k} < \int_{0}^{a_k} A ((f\chi_{I_k})^*(s))\, d \mathcal L^1 (s)
=\int_{I_k} A(|f(s)|)\,d \mathcal L^1 (s)
\end{equation}
for every $k \in \N$, one has that
\begin{equation}\nonumber
\sum_{k=1}^\infty a_k  <  \sum_{k=1}^\infty \int_{I_k} A(|f(s)|)\,d
\mathcal L^1 (s) \leq  \int_{0}^1A (|f(s)|)\, d \mathcal L^1 (s) <
\infty.
\end{equation}
Notice that the last inequality holds owing to the assumption that
$A$  satisfies the $\Delta_2$-condition near infinity, and $f$ has
support of finite measure. Altogether, condition \textup{(H)} is
satisfied by the norm $\|\cdot\|_{L^A(\mathbb R^n)}$.
\end{proof}

\begin{proof}[Proof of Proposition~\ref{T:Orlicz}.] If $A$ satisfies
the $\Delta _2$-condition near infinity, then the norm
$\|\cdot\|_{L^A(\Rn)}$ fulfills the Lebesgue point property, by
Lemma~\ref{L:orlicz} and  Theorem~\ref{T:main-theorem3}. Conversely,
assume that the norm $\|\cdot\|_{L^A(\Rn)}$ fulfills the Lebesgue
point property. Then it has to be locally absolutely continuous, by
either Theorem~\ref{T:main-theorem1} or Theorem~\ref{T:main-theorem2}. Owing to
\cite[Theorem 14 and Corollary 5, Section 3.4]{RR}, this implies
that  $A$ satisfies the $\Delta _2$-condition near infinity.
\end{proof}

In the next proposition, we point
out  the property, of independent interest,  that the functional
$\mathcal G_{L^A}$ is almost concave for \textit{any} $N$-function
$A$. Such a  property, combined with the fact that the norm
$\|\cdot\|_{L^A(\mathbb R^n)}$ is locally absolutely continuous if
and only if $A$ satisfies the $\Delta _2$-condition near infinity,
leads to an alternative proof of Proposition~\ref{T:Orlicz}, at
least when $A$ is an $N$-function, via
Theorem~\ref{T:main-theorem1}.

\begin{proposition}\label{orliczconcave}
The functional $\mathcal G_{L^A}$ is  almost concave for every
$N$-function $A$.
\end{proposition}
\begin{proof} The norm $\|\,\cdot \,
\|_{L^A(\Rn)}$ is equivalent, up to multiplicative constants, to the
norm $\| \cdot \|_{L_A(\Rn)}$
defined as
$$\|u \|_{L_A(\Rn)}= \inf \bigg\{\frac 1k\bigg( 1 + \int _{\Rn} A(k|u(x)|)\,
dx\bigg) : k>0\bigg\}$$
 for $u \in L^0(\Rn )$ -- see \cite[Section 3.3, Proposition 4 and Theorem
 13]{RR}. One has that
$$\frac 1k\bigg( 1 + \int _{\Rn} A(k|u(x)|)\,
dx\bigg) = \frac 1k  + \int _0^\infty  A'(kt) u_*(t)\, dt,$$
 where $A'$ denotes  the left-continuous derivative of
 $A$. Altogether, we have that
 $$\mathcal G_{L_A} (f) = \inf \bigg\{\frac 1k  + \int _0^\infty  A'(kt) f(t)\,
 dt\bigg\}$$
 for every non-increasing function $f : [0, \infty ) \to [0,
 \infty)$. The functional $\mathcal G_{L_A}$ is concave, since it is
 the infimum of a family of linear functionals, and hence the
 functional $\mathcal G_{L^A}$ is almost concave.
  \end{proof}

Let us next focus on the case of Lorentz  endpoint norms, which is
the object of Proposition~\ref{T:Lambda}.

\begin{lemma}\label{L:concavity-lorentz-endpoint}
Assume that $\varphi : [0, \infty) \to [0, \infty)$ is a (non
identically vanishing) concave function.
%
Then
\begin{equation}\label{LambdaG}
\G_{\Lambda_\varphi}(f)=\int_0^{\mathcal L^1(\{ f>0 \})}
\varphi(f(t))\,d \mathcal L^1 (t)
\end{equation}
for every
 non-increasing function $f: [0,\infty)
\rightarrow [0,\infty]$. In particular, the functional
$\G_{\Lambda_\varphi}$ is concave.
\end{lemma}

\begin{proof} Take any non-increasing function $f:[0,\infty)\rightarrow [0,\infty]$. Set $h=f_*$, whence
 $f=f^*=(f_*)_*= h_*$ a.e., and  $h^*(0) = f_*(0) = \mathcal L^1(\{ f>0 \})$.
 From equations \eqref{Lorentzbis} and  \eqref{easy G_X}, one has,  via Fubini's
theorem,
\begin{align*}
\G_{\Lambda_\varphi}(f) &= \G_{\Lambda_\varphi}(h_*)=\|h\|_{\Lambda_\varphi (0, \infty)}  = h^*(0) \varphi
(0^+) + \int _0^\infty h^*(s) \varphi '(s)\, d\mathcal L^1(s) \\ \nonumber & =
h^*(0) \varphi (0^+) + \int _0^\infty \int
_0^{h^*(s)}d\mathcal L^1(t)\, \varphi '(s)\, d\mathcal L^1(s)
\\ \nonumber & =
h^*(0) \varphi (0^+) + \int _0^{h^*(0)} \int _0^{h_*(t)} \varphi '(s)\, d\mathcal L^1(s)\, d\mathcal L^1(t)
\\ \nonumber & =
h^*(0) \varphi (0^+) + \int _0^{h^*(0)}  [\varphi (h_*(t)) - \varphi (0^+)]\, d\mathcal L^1(t)
\\ \nonumber & =
 \int _0^{h^*(0)}  \varphi (h_*(t)) \, d\mathcal L^1(t) =
 \int_0^{\mathcal L^1(\{ f>0 \})}
\varphi(f(t))\,d \mathcal L^1 (t).
\end{align*}
Hence, formula \eqref{LambdaG} follows.
\\ In order to verify the
concavity of $\G_{\Lambda_\varphi}$, fix any pair of non-increasing
functions $f, g: [0,\infty) \rightarrow [0,\infty]$ and $\lambda \in
(0,1)$. Observe that
$$
\{t\in [0,\infty): \lambda f(t)+(1-\lambda)g(t)>0\} = \{t\in
[0,\infty): f(t)>0\} \cup \{t\in [0,\infty): g(t)>0\}.
$$
The monotonicity of $f$ and $g$ ensures that the two sets on the
right-hand side of the last equation are intervals whose left
endpoint is $0$. Consequently,
\begin{equation}\label{august100} \mathcal L^1(\{\lambda f
+(1-\lambda)g>0\}) = \max\{\mathcal L^1(\{ f >0\}), \mathcal L^1(\{
g >0\})\}.
\end{equation}
%
On making use of equations \eqref{LambdaG} and \eqref{august100},
and of  the concavity of $\varphi$, one infers that the functional
$\G_{\Lambda_\varphi}$ is concave as well.
\end{proof}

\medskip

\begin{proof}[Proof of Proposition~\ref{T:Lambda}] By Lemma~\ref{L:concavity-lorentz-endpoint}, the functional
$\G_{\Lambda_\varphi}$ is  concave
for every non identically vanishing concave function $\varphi : [0, \infty) \to [0, \infty)$.
On the other hand, it is easily verified, via equation
\eqref{Lorentzbis}, that the norm $\| \cdot \|_{\Lambda_\varphi
(\Rn)}$ is locally absolutely continuous if, and only if, $\varphi
(0^+)=0$. The conclusion thus follows from
Theorem~\ref{T:main-theorem1}.
\end{proof}

\medskip

We conclude with a proof of Proposition~\ref{T:Marcinkiewicz}.

\begin{proof}[Proof of Proposition~\ref{T:Marcinkiewicz}]
Assume first that $\lim_{s\to
0^+}\tfrac s{\varphi(s)}=0$. Then we claim that the norm
$\|\cdot\|_{{M_\varphi}(\Rn)}$ is not locally absolutely continuous,
and hence, by either Theorem~\ref{T:main-theorem1} or
Theorem~\ref{T:main-theorem2}, it does not have the Lebesgue point
property. To verify this claim, observe that the function $(0,
\infty) \ni s \mapsto \tfrac s{\varphi(s)}$ is quasiconcave in the
sense of \cite[Definition 5.6, Chapter 2]{BS}, and hence, by
\cite[Chapter~2, Proposition~5.10]{BS}, there exists a~concave
function $\psi : (0, \infty) \to [0, \infty)$ such that $\tfrac 12
\psi (s) \leq \tfrac s{\varphi(s)}\leq \psi (s)$ for $s \in (0,
\infty)$. Let $\psi '$ denote the right-continuous derivative of
$\psi$, and define $u(x)=\psi'(\omega_n|x|\sp n)$ for $x\in\mathbb
R\sp n$, where $\omega_n$ is the volume of the
   unit ball in $\mathbb R\sp n$. Then $u\sp*=\psi'$ in
   $(0,\infty)$, so that
   $$1 \leq u^{**}(s) \varphi (s) \leq 2 \quad \hbox{for $s \in (0,
   \infty)$.}$$
   The second inequality in the last equation ensures that   $u\in M_\varphi(\Rn)$, whereas the first one tells us that $u$ does not
   have a
    locally absolutely continuous norm in $M_\varphi(\Rn)$.
 \\Conversely, assume that $\lim_{s\to 0^+}\tfrac s{\varphi(s)}>0$, then $(M_{\varphi})_{\loc}(\Rn)= L^1_{\loc}(\Rn)$, with equivalent norms
 on any given subset of $\Rn$ with finite measure (see e.g.
 \cite[Theorem~5.3]{S}).
 Hence, the  norm $\|\cdot\|_{{M_\varphi}(\Rn)}$ has the Lebesgue point property, since
 $\|\cdot\|_{L^1(\Rn)}$ has it.
\end{proof}


\end{document}